\documentclass[fleqn]{amsart}

\usepackage{amsmath,amsthm,amssymb,latexsym}
\usepackage{url}

\newtheorem{theorem}{Theorem}[section]
\newtheorem{definition}[theorem]{Definition}
\newtheorem{lemma}[theorem]{Lemma}
\newtheorem{fact}[theorem]{Fact}
\newtheorem{claim}{Claim}
\newtheorem{prop}[theorem]{Proposition}
\newtheorem{cor}[theorem]{Corollary}

\newcommand{\mrm}{\mathrm}
\newcommand{\mbb}{\mathbb}
\newcommand{\mcal}{\mathcal}
\newcommand{\msf}{\mathsf}
\newcommand{\lchon}{\textrm{``}}
\newcommand{\rchon}{\textrm{''}}
\newcommand{\rst}{\!\upharpoonright\!}
\newcommand{\dmnd}{\diamondsuit}

\newcommand{\suc}{\mathop{\mathrm{suc}} \nolimits}
\newcommand{\dom}{\mathop{\mathrm{dom}} \nolimits}
\newcommand{\ran}{\mathop{\mathrm{ran}} \nolimits}
\newcommand{\nart}{\mathop{\mathrm{rt}} \nolimits}

\newenvironment{renumerate}%
{\begin{enumerate}}{\end{enumerate}}
\newenvironment{Renumerate}%
{\begin{enumerate}}{\end{enumerate}}
\newenvironment{aenumerate}%
{\begin{enumerate}}{\end{enumerate}}

\title{An extension of Subcomplete Forcing Axiom which implies $\dmnd^+$}
\author{Hiroshi Sakai}
\address{Graduate School of Mathematical Sciences, The University of Tokyo, Tokyo 153-8902, Japan}
\email{hrshsakai@g.ecc.u-tokyo.ac.jp}
\thanks{This work is supported by JSPS KAKENHI Grant Numbers 21K03338, 24K06828}

\begin{document}

\begin{abstract}
We prove that the Subcomplete Forcing Axiom ($\msf{SCFA}$), introduced by Jensen \cite{Jensen3},
is consistent with $\dmnd_{\omega_1}^+$.
For this, we introduce a weaker variation of the subcompleteness of forcing notions.
More precisely, for a certain kind of $\dmnd_{\omega_1}$-like sequence $\vec{K}$,
we define the notion of $\vec{K}$-subcompleteness, which is a weaker notion than subcompleteness.
Then, the forcing axiom for $\vec{K}$-subcomplete forcing notions ($\vec{K}$-$\msf{SCFA}$)
implies $\msf{SCFA}$. We show that $\vec{K}$-$\msf{SCFA}$
is consistent and implies $\dmnd_{\omega_1}^+$.
\end{abstract}

\maketitle


\section{Introduction}

Jensen \cite{Jensen1,Jensen4} introduced the notion of subcomplete forcing,
which is weaker than $\sigma$-closedness.
Besides $\sigma$-closed forcing notions, the class of all subcomplete forcing notions includes
Namba forcing (under the Continuum Hypothesis $\msf{CH}$),
Prikry forcing and forcing notions shooting clubs through stationary subsets of ordinals of
countable cofinality.
Subcomplete forcings add no new reals and preserve stationary subsets of $\omega_1$.
Also, all revised countable support iterations of subcomplete forcings are subcomplete.
Jensen \cite{Jensen2} used subcomplete forcing notions
to settle the consistency strength of the Extended Namba Problem.

The original notion of subcompleteness in \cite{Jensen1,Jensen2} included an additional technical condition
needed for its preservation under revised countable support iterations.
Recently, Fuchs-Switzer \cite{FS} showed that this technical condition is not necessary for Miyamoto's nice iterations \cite{Miyamoto_nice}.

The forcing axiom for subcomplete forcing notions,
which is called the Subcomplete Forcing Axiom and denoted as $\msf{SCFA}$,
is interesting as a fragment of Martin's Maximum ($\msf{MM}$) consistent with $\msf{CH}$.
In fact, Jensen \cite{Jensen3} proved that $\msf{SCFA}$ is consistent with $\diamondsuit_{\omega_1}$.
Jensen \cite{Jensen3} also proved that $\msf{SCFA}$ implies several interesting consequences of $\msf{MM}$
such as the Singular Cardinal Hypothesis and the reflection of stationary sets consisting of ordinals of countable cofinality.
Consequences of $\msf{SCFA}$ are investigated in more detail in \cite{Fuchs,Fuchs2,Fuchs3,FM,FS,SS}.

Besides these consequences of $\msf{SCFA}$,
$\msf{MM}$ has many interesting consequences which are consistent with $\dmnd_{\omega_1}$.
For example, $\msf{MM}$ implies the reflection of stationary subsets of $\mcal{P}_{\omega_1} ( \lambda )$
for $\lambda \geq \omega_2$, which is often called the Weak Reflection Principle and denoted as $\msf{WRP}$,
Chang's Conjecture and the non-existence of $\omega_1$-Kurepa trees.
It was asked in \cite{Fuchs} and \cite{FS} whether these consequences of $\msf{MM}$ also
follow from $\msf{SCFA}$. Recall that
\[
\msf{WRP} \ \Rightarrow \ \mbox{Chang's Conjecture} \ \Rightarrow \ 
\neg \exists \, \mbox{$\omega_1$-Kurepa trees} \ \Rightarrow \ \neg \diamondsuit^+_{\omega_1} \, .
\]

In this paper, we answer this question negatively.
Namely, we prove that if $\msf{ZFC}$ is consistent with the existence of a supercompact cardinal,
then $\msf{ZFC} + \msf{SCFA} + \diamondsuit^+_{\omega_1}$ is consistent. (Corollary \ref{cor:con_scfa_dmnd+})

To prove this, we introduce the notion of $\vec{K}$-subcomplete forcing
for a $\dmnd$-model sequence $\vec{K} = \langle K_\xi \mid \xi < \omega_1 \rangle$.
It will follow from the definition that all subcomplete forcing notions are $\vec{K}$-subcomplete.
So the forcing axiom for $\vec{K}$-subcomplete forcing notions, denoted as $\vec{K}$-$\msf{SCFA}$,
implies $\msf{SCFA}$.
We prove the following.
\begin{Renumerate}
\item For any $\dmnd$-model sequence $\vec{K}$,
all nice iterations of $\vec{K}$-subcomplete forcings are $\vec{K}$-subcomplete.
(Theorem \ref{thm:K_subcomplete_nice_it})
\item If $\msf{ZFC}$ is consistent with the existence of a supercompact cardinal,
then it is consistent with $\vec{K}$-$\msf{SCFA}$ for some $\dmnd$-model sequence $\vec{K}$.
(Theorem \ref{thm:con_K_scfa})
\item If $\vec{K}$-$\msf{SCFA}$ holds for some $\dmnd$-model sequence $\vec{K}$,
then $\msf{SCFA}$ and $\dmnd^+_{\omega_1}$ hold.
(Proposition \ref{prop:K_scfa_scfa} and \ref{prop:K_scfa_dmnd+})
\end{Renumerate}

This paper is organized as follows.
In \S \ref{sec:preliminaries}, we present our notation and basic facts used in this paper.
In \S \ref{sec:K_subcomplete}, we introduce the notion of $\vec{K}$-subcomplete forcing
and study its basic properties.
In \S \ref{sec:iteration}, we discuss nice iterations of $\vec{K}$-subcomplete forcings to prove (I) above.
Finally, in \S \ref{sec:K_scfa}, we investigate $\vec{K}$-$\msf{SCFA}$ to prove (II) and (III).


\section{Preliminaries} \label{sec:preliminaries}

In this section, we present our notation and basic facts used in this paper.
For background material not covered here, see Kunen \cite{Kunen} or Jech \cite{Jech}.


\subsection{$\dmnd$-principles}

In this paper, we deal with several $\dmnd$-principles. Here we recall them.

\begin{itemize}
\item A $\dmnd_{\omega_1}$-\emph{sequence} is a sequence $\langle b_\xi \mid \xi < \omega_1 \rangle$
such that for any $B \subseteq \omega_1$ there are stationary many $\xi < \omega_1$ with $B \cap \xi = b_\xi$.
\item A $\dmnd_{\omega_1}^-$-\emph{sequence} is a sequence $\langle \mcal{B}_\xi \mid \xi < \omega_1 \rangle$
of countable sets such that
for any $B \subseteq \omega_1$ there are stationary many $\xi < \omega_1$ with $B \cap \xi \in \mcal{B}_\xi$.
\item A $\dmnd_{\omega_1}^*$-\emph{sequence} is a sequence $\langle \mcal{B}_\xi \mid \xi < \omega_1 \rangle$
of countable sets such that
for any $B \subseteq \omega_1$ there are club many $\xi < \omega_1$ with $B \cap \xi \in \mcal{B}_\xi$.
\item A $\dmnd_{\omega_1}^+$-\emph{sequence} is a sequence $\langle \mcal{B}_\xi \mid \xi < \omega_1 \rangle$
of countable sets such that
for any $B \subseteq \omega_1$ there is a club $C \subseteq \omega_1$ with
$B \cap \xi , C \cap \xi \in \mcal{B}_\xi$ for all $\xi \in C$.
\end{itemize}

Let $\dmnd_{\omega_1}$ ($\dmnd_{\omega_1}^-$, $\dmnd_{\omega_1}^*$, $\dmnd_{\omega_1}^+$, respectively)
be the assertion that a $\dmnd_{\omega_1}$-sequence (a $\dmnd_{\omega_1}^-$-sequence,
a $\dmnd_{\omega_1}^*$-sequence, a $\dmnd_{\omega_1}^+$-sequence, respectively) exists.
We have the following implication relations.
\[
\dmnd_{\omega_1}^+ \ \Rightarrow \ \dmnd_{\omega_1}^* \ \Rightarrow \ 
\dmnd_{\omega_1}^- \ \Leftrightarrow \ \dmnd_{\omega_1}
\]
Recall also that $\dmnd_{\omega_1}^+$ implies the existence of an $\omega_1$-Kurepa tree.
See Kunen \cite{Kunen} for proofs of these facts.


\subsection{Forcing and its iteration}

In this paper, we follow Miyamoto \cite{Miyamoto_nice} for notations on forcing.
A \emph{forcing notion} is a separative preorder with a largest element.
Recall that a preorder $\mbb{P}$ is separative if for any $p,q \in \mbb{P}$ with
$p \nleq_\mbb{P} q$, there is $p' \leq_\mbb{P} p$ which is incompatible with $q$.
Recall also that a preorder $\mbb{P}$ is separative if and only if
for any $p,q \in \mbb{P}$, $p \leq_\mbb{P} q$ exactly when
$p \Vdash_\mbb{P} \lchon\, q \in \dot{G} \,\rchon$, where $\dot{G}$ is the canonical name
for a $\mbb{P}$-generic filter.

Let $\mbb{P}$ be a forcing notion.
The largest element of $\mbb{P}$ is denoted as $1_\mbb{P}$.
For $p,q \in \mbb{P}$, we let $p \equiv_\mbb{P} q$ denote that $p \leq_\mbb{P} q$ and $q \leq_\mbb{P} p$.
A subscript $\mbb{P}$ in $\leq_\mbb{P}$, $\Vdash_\mbb{P}$, $1_\mbb{P}$
and $\equiv_\mbb{P}$ is often omitted when it is clear from the context.

For a forcing notion $\mbb{P}$, \emph{the completion} of $\mbb{P}$ is a complete Boolean algebra $\mbb{B}$
such that there is a dense embedding from $\mbb{P}$ to $\mbb{B} \setminus \{ 0_\mbb{B} \}$.
For any forcing notion, its completion exists and is unique up to isomorphism.
For forcing notions $\mbb{P}$ and $\mbb{Q}$, we say that $\mbb{P}$ and $\mbb{Q}$ are \emph{forcing equivalent} if
the completions of $\mbb{P}$ and $\mbb{Q}$ are isomorphic.

Next, we present notations on forcing iterations.
A sequence $\langle \mbb{P}_\alpha \mid \alpha < \delta \rangle$ of forcing notions,
where $\delta \in \mrm{On}$,
is called an \emph{iteration} if, for all $\alpha , \beta$ with $\alpha \leq \beta < \delta$,
the following conditions hold,
where $\leq_\alpha$ and $1_\alpha$ denote $\leq_{\mbb{P}_\alpha}$ and $1_{\mbb{P}_\alpha}$,
respectively.
\begin{renumerate}
\item $\mbb{P}_\alpha$ consists of functions on $\alpha$.
\item For all $p \in \mbb{P}_\beta$ we have $p \rst \alpha \in \mbb{P}_\alpha$,
and $1_\beta \rst \alpha = 1_\alpha$.
\item For any $p \in \mbb{P}_\alpha$ and any $q \in \mbb{P}_\beta$,
if $p \leq_\alpha q \rst \alpha$, then $p^\frown q \rst [ \alpha , \beta ) \in \mbb{P}_\beta$,
and $p^\frown q \rst [ \alpha , \beta ) \leq_\beta q$.
\item For any $p , q \in \mbb{P}_\beta$, if $p \leq_\beta q$, then $p \rst \alpha \leq_\alpha q \rst \alpha$,
and $p \leq_\beta p \rst \alpha^\frown q \rst [ \alpha , \beta )$.
\item If $\alpha$ is a limit ordinal, then for any $p,q \in \mbb{P}_\alpha$,
$p \leq_\alpha q$ if and only if $p \rst \gamma \leq_\gamma q \rst \gamma$ for all $\gamma < \alpha$.
\end{renumerate}

For an iteration $\langle \mbb{P}_\alpha \mid \alpha < \delta \rangle$,
we let $\leq_\alpha$, $1_\alpha$ and $\Vdash_\alpha$ denote $\leq_{\mbb{P}_\alpha}$,
$1_{\mbb{P}_\alpha}$ and $\Vdash_{\mbb{P}_\alpha}$, respectively.
Also, let $\dot{G}_\alpha$ denote the canonical name for a $\mbb{P}_\alpha$-generic filter.
A subscript $\alpha$ in $\leq_\alpha$, $1_\alpha$ and $\Vdash_\alpha$ is sometimes omitted
if it is clear from the context.
For $p \in \mbb{P}_\alpha$, $\alpha = \dom (p)$ will be denoted as $l(p)$.

Suppose $\langle \mbb{P}_\alpha \mid \alpha < \delta \rangle$ is an iteration,
$\alpha \leq \beta < \delta$, and $G_\alpha$ is a $\mbb{P}_\alpha$-generic filter over $V$.
In $V[ G_\alpha ]$, let $\mbb{P}_{\alpha , \beta}$ be the following forcing notion,
where $\leq_{\alpha , \beta}$ and $1_{\alpha , \beta}$ are its order and largest element, respectively.
\begin{renumerate}
\item $\mbb{P}_{\alpha , \beta} :=
\{ p \rst [ \alpha , \beta ) \mid p \in \mbb{P}_\beta \,\wedge\, p \rst \alpha \in G_\alpha \}$.
\item $p \leq_{\alpha , \beta} q$ if there are $p' , q' \in G_\alpha$ such that
$p'^\frown p \leq_\beta q'^\frown q$.
\item $1_{\alpha , \beta} := 1_\beta \rst [ \alpha , \beta )$.
\end{renumerate}
Then, $\mbb{P}_{\alpha , \beta}$ is a forcing notion in $V[ G_\alpha ]$.
(See \cite[Prop.~1.2]{Miyamoto_nice}.)
Let $\dot{\mbb{P}}_{\alpha , \beta}$ be a $\mbb{P}_\alpha$-name for $\mbb{P}_{\alpha , \beta}$.

If $G_\beta$ is a $\mbb{P}_\beta$-generic filter over $V$,
then $G_\beta \rst \alpha := \{ p \rst \alpha \mid p \in G_\beta \}$ is a $\mbb{P}_\alpha$-generic
filter over $V$, and $G_\beta \rst [ \alpha , \beta ) := \{ p \rst [ \alpha , \beta ) \mid p \in G_\beta \}$
is a $( \dot{\mbb{P}}_{\alpha , \beta} )^{G_\beta \upharpoonright \alpha}$-generic filter
over $V[ G_\beta \rst \alpha ]$.
Also, if $G_\alpha$ is a $\mbb{P}_\alpha$-generic filter over $V$, and
$H$ is a $( \dot{\mbb{P}}_{\alpha , \beta} )^{G_\alpha}$-generic filter over $V$, then
\[
G_\alpha * H := \{ p \in \mbb{P}_\beta \mid p \rst \alpha \in G_\alpha \,\wedge\, p \rst [ \alpha , \beta ) \in H \}
\]
is a $\mbb{P}_\beta$-generic filter over $V$. (See \cite[Prop.~1.3]{Miyamoto_nice}.)
So $\mbb{P}_\beta$ is forcing equivalent to $\mbb{P}_\alpha * \dot{\mbb{P}}_{\alpha , \beta}$.
In particular, $\mbb{P}_{\alpha + 1}$ is forcing equivalent to
$\mbb{P}_\alpha * \dot{\mbb{P}}_{\alpha , \alpha + 1}$.


\subsection{Subcomplete forcing}

We briefly review the notion of subcomplete forcing.
We recall a simplified version of the original definition, called \(\infty\)-subcompleteness, introduced by Fuchs-Switzer \cite{FS}.

The notion of ($\infty$-)subcomplete forcing involves models of set theory.
First, we give our notation on them.

As is usual, for a first order structure $M$, we often identify $M$ with its universe.
For example, if we write $a \in M$, then it means that $a$ belongs to the universe of $M$,
and if we say that $M$ is countable, then it means that the universe of $M$ is countable.

For structures $M$ and $N$ of the same language and for a function $\sigma : M \to N$,
let $\sigma : M \prec N$ denote that $\sigma$ is an elementary embedding from $M$ to $N$.

Let $\mcal{L}_{\mrm{ST}} = \{ \tilde{\in} \}$ be the language of Set Theory,
where $\tilde{\in}$ is a binary relation symbol for the $\in$-relation.
Let $\mcal{L}_{\mrm{ST}}^+ := \{ \tilde{\in} , \tilde{P} \}$, where $\tilde{P}$ is a unary predicate symbol.

Suppose $M$ is an $\mcal{L}_{\mrm{ST}}$-structure $\langle X , E \rangle$
or an $\mcal{L}_{\mrm{ST}}^+$-structure $\langle X , E , P \rangle$.
We say that $M$ is a model of $\mrm{ZF}^-$ ($\mrm{ZFC}^-$) if
$M$ satisfies all axioms of $\mrm{ZF}$ ($\mrm{ZFC}$) except for the Power Set Axiom.
Here, if $M$ is an $\mcal{L}_{\mrm{ST}}^+$-structure, then the Axiom Schemes of Replacement and Separation
are applied to all $\mcal{L}_{\mrm{ST}}^+$-formulas.
We say that $M$ is \emph{transitive} if $X$ is a transitive set, and $E = {\in} \cap ( X \times X )$.

For a set $A$ and an ordinal $\chi$, let $L^A_\chi$ be $\langle L_\chi [A] , {\in} , A \cap L_\chi [A] \rangle$.
Note that $L^A_\chi$ is a transitive $\mcal{L}_\mrm{ST}^+$-structure.

Suppose $M$ is a transitive $\mcal{L}_{\mrm{ST}}^+$-model of $\mrm{ZFC}^-$.
$M$ is said to be \emph{full} if there is a transitive $\mcal{L}_{\mrm{ST}}$-model $N$ of $\mrm{ZF}^-$
such that $M \in N$, and $M$ is regular in $N$, where we say that $M$ is \emph{regular} in $N$
if for any $x \in M$ and any function $f : x \to M$ with $f \in N$, we have $f \in M$.

Suppose $M = \langle X , {\in} , P \rangle$ is a transitive $\mcal{L}_{\mrm{ST}}^+$-model of $\msf{ZFC}^-$,
and $\mbb{P}$ is a forcing notion in $M$.
For a $\mbb{P}$-generic filter $G$ over $M$, let $M[G]$ denote the $\mcal{L}_{\mrm{ST}}^+$-structure
$\langle X[G] , {\in} , P \rangle$, where
$X [G] = \{ \dot{a}^G \mid \mbox{$\dot{a}$ is a $\mbb{P}$-name in $X$} \}$.
Note that $M[G]$ is a transitive $\mcal{L}_{\mrm{ST}}^+$-model of $\mrm{ZFC}^-$.
Note also that if $M$ is full, then $M[G]$ is also full.
In fact, if $N$ witnesses the fullness of $M$,
then $G$ is $\mbb{P}$-generic filter over $N$, and $N[G]$ witnesses the fullness of $M[G]$.

Now, we recall the notion of subcompleteness.
As we mentioned above, we recall a simplified version, which is called the $\infty$-subcompleteness,
introduced by Fuchs-Switzer \cite{FS}.

\begin{definition}[Fuchs-Switzer \cite{FS}] \label{def:infty_subcomplete}
Suppose $\mbb{P}$ is a forcing notion.

For a regular cardinal $\theta$ and $a \in \mcal{H}_\theta$,
we say that $\theta$ and $a$ \emph{verify the} $\infty$-\emph{subcompleteness} of $\mbb{P}$
if $\mbb{P} \in \mcal{H}_\theta$, and the following hold:
For any $A$, $\chi$, $\bar{M}$, $\bar{\mbb{P}}$, $\bar{b}$, $\sigma$, $b$ and $\bar{G}$, if
\begin{renumerate}
\item $A$ is a set, $\chi$ is an ordinal, and $\mcal{H}_\theta \subseteq L^A_\chi \models \mrm{ZFC}^-$,
\item $\bar{M}$ is a countable transitive full $\mcal{L}_{\mrm{ST}}^+$-model of $\msf{ZFC}^-$
with $\bar{\mbb{P}} , \bar{b} \in \bar{M}$,
\item $\sigma : \bar{M} \prec L^A_\chi$,  $a \in \ran ( \sigma )$ and
$\sigma ( \langle \bar{\mbb{P}} , \bar{b} \rangle ) = \langle \mbb{P} , b \rangle$,
\item $\bar{G}$ is a $\bar{\mbb{P}}$-generic filter over $\bar{M}$,
\end{renumerate}
then there is $p^* \in \mbb{P}$ which forces that there is $\sigma^* : \bar{M} \prec L^A_\chi$ with
$\sigma^* ( \langle \bar{\mbb{P}} , \bar{b} \rangle ) = \langle \mbb{P} , b \rangle$
and $\sigma^* [ \bar{G} ] \subseteq \dot{G}$,
where $\dot{G}$ is the canonical name for a $\mbb{P}$-generic filter.

We say that $\mbb{P}$ is $\infty$-\emph{subcomplete} if
there is a regular cardinal $\theta$ which verifies the $\infty$-subcompleteness of $\mbb{P}$.
\end{definition}

The above definition of $\infty$-subcompleteness is slightly different from the one in \cite[Definition 3.2]{FS},
which does not use the parameter $a \in \mcal{H}_\theta$. But, as is mentioned after Definition 3.2 in \cite{FS},
these definitions are equivalent. See also Jensen \cite[Chapter 3, Lemma 2.5]{Jensen4}
and Fuchs \cite[Observation 2.27]{Fuchs3}.

The original subcompleteness in Jensen \cite{Jensen1,Jensen2} has some additional condition
requiring $\sigma^*$ to have some similarity to $\sigma$.
This condition is omitted in the $\infty$-subcompleteness. So the $\infty$-subcompleteness
is weaker than the original subcompleteness.
This additional condition is used to show that all revised countable support iterations
of subcomplete forcings are subcomplete. See \cite{Jensen1,Jensen2} for details.
Fuchs-Switzer \cite{FS} proved that this condition is not necessary for nice iterations
developed by Miyamoto \cite{Miyamoto_nice}. Namely, they proved the following.

\begin{fact}[Fuchs-Switzer \cite{FS}] \label{fact:infty_subcomplete_nice_it}
All nice iterations of $\infty$-subcomplete forcings are $\infty$-subcomplete.
\end{fact}

\noindent
It should be noted here that a similar result for subproper forcings was obtained
by Miyamoto \cite{Miyamoto_subproper} before.

Also, as far as we know, $\infty$-subcomplete forcings have all important properties of subcomplete forcings.
For example, $\infty$-subcomplete forcings add no reals, preserve stationary subsets of $\omega_1$
and preserve $\diamondsuit_{\omega_1}$.

Now, we turn our attention to the forcing axiom for $\infty$-subcomplete forcings.
By Fact \ref{fact:infty_subcomplete_nice_it} and the standard argument for the consistency proof
of forcing axioms, we can prove its consistency.

\begin{definition} \label{def:infty_scfa}
\emph{The ($\infty$-)Subcomplete Forcing Axiom}, denoted as $($$\infty$-$)$$\msf{SCFA}$,
is the following assertion.
\begin{quote}
For any ($\infty$-)subcomplete forcing notion $\mbb{P}$
and any family $\mcal{D}$ of dense subsets of $\mbb{P}$ with $| \mcal{D} | \leq \omega_1$,
there is a filter $g$ on $\mbb{P}$ such that $g \cap D \neq \emptyset$ for any $D \in \mcal{D}$.
\end{quote}
\end{definition}

\begin{fact}[Fuchs-Switzer \cite{FS}] \label{fact:infty_scfa}
Assume there is a supercompact cardinal.
Then there is a forcing extension in which $\infty$-$\msf{SCFA}$ holds.
\end{fact}

In fact, since $\infty$-subcomplete forcings preserve $\diamondsuit_{\omega_1}$,
if $\diamondsuit_{\omega_1}$ holds in the ground model, then $\diamondsuit_{\omega_1}$ also holds in the extension.
Note also $\infty$-$\msf{SCFA}$ implies $\msf{SCFA}$
since all subcomplete forcing notions are $\infty$-subcomplete.


\section{$\vec{K}$-subcomplete forcing} \label{sec:K_subcomplete}

In this section, we introduce the notion of $\vec{K}$-subcomplete forcing and study its basic properties.
The notion of $\vec{K}$-subcomplete forcing originates from the proof, due to Jensen \cite{Jensen3}, of the fact that
subcomplete forcings preserve $\diamondsuit_{\omega_1}$.
It is defined for an adequate model sequence $\vec{K}$,
which is a guessing sequence on $\omega_1$.
Roughly speaking, it is obtained by restricting $\bar{M}$ and $\bar{G}$
in the $\infty$-subcompleteness to those captured by $\vec{K}$.

In \S \ref{subsec:adequate}, we introduce the notion of adequate model sequences
and study its basic properties.
In \S \ref{subsec:K_subcomplete}, we introduce the notion of $\vec{K}$-subcomplete forcings
and study its basic properties.


\subsection{Adequate model sequences} \label{subsec:adequate}

First, we introduce adequate model sequences and $\dmnd$-model sequences.
A similar model sequence was also used in Jensen \cite{Jensen3} to prove the preservation of $\diamondsuit_{\omega_1}$
by subcomplete forcings.

\begin{definition} \label{def:model_seq}
A sequence $\vec{K} = \langle K_\xi \mid \xi < \omega_1 \rangle$ is called
an \emph{adequate model sequence} if
\begin{renumerate}
\item for each $\xi < \omega_1$, $K_\xi$ is a transitive $\mcal{L}_{\mrm{ST}}$-model $\mrm{ZFC}^-$,
$\xi \in K_\xi$, and $\xi$ is countable in $K_\xi$,
\item for any $B \subseteq \omega_1$, there are stationary many $\xi < \omega_1$
with $B \cap \xi \in K_\xi$.
\end{renumerate}
An adequate model sequence $\vec{K} = \langle K_\xi \mid \xi < \omega_1 \rangle$
is called a $\dmnd$-\emph{model sequence} if
\begin{renumerate}
\addtocounter{enumi}{2}
\item $K_\xi$ is countable for each $\xi < \omega_1$.
\end{renumerate}
\end{definition}

Note that if $\vec{K} = \langle K_\xi \mid \xi < \omega_1 \rangle$ satisfies (i) of the above definition,
and $\mcal{H}_{\omega_1} \subseteq K_\xi$ for all $\xi < \omega_1$,
then $\vec{K}$ is an adequate model sequence.

The $\vec{K}$-subcompleteness is defined for an adequate model sequence $\vec{K}$.
Before we give its definition, we study basic properties of adequate model sequences.

\begin{lemma} \label{lem:model_seq_basic}
\begin{aenumerate}
\item A $\dmnd$-model sequence exists if and only if $\dmnd_{\omega_1}$ holds.
\item Suppose $\vec{K} = \langle K_\xi \mid \xi < \omega_1 \rangle$ is an adequate model sequence,
and let $F$ be the set of all $D \subseteq \omega_1$ such that
$\{ \xi < \omega_1 \mid B \cap \xi \in K_\xi \} \cap C \subseteq D$ for some $B \subseteq \omega_1$
and some club $C \subseteq \omega_1$.
Then, $F$ is a normal filter over $\omega_1$.
\end{aenumerate}
\end{lemma}

\begin{proof}
(1) If there is a $\dmnd$-model sequence $\vec{K}$,
then $\vec{K}$ witnesses $\dmnd_{\omega_1}^-$, which is equivalent to $\dmnd_{\omega_1}$.
Conversely, suppose $\dmnd_{\omega_1}$ holds. Let $\langle b_\xi \mid \xi < \omega_1 \rangle$
be a $\dmnd_{\omega_1}$-sequence, and for each $\xi < \omega_1$ take a countable $K_\xi$
such that $b_\xi , \xi \in K_\xi \prec \langle \mcal{H}_{\omega_1} , {\in} \rangle$.
Then $\langle K_\xi \mid \xi < \omega_1 \rangle$ is a $\dmnd$-model sequence.

\medskip

\noindent
(2) We only prove the normality of $F$. The other properties are easily checked.

Suppose $\{ D_\eta \mid \eta < \omega_1 \} \subseteq F$.
We show that $D := \Delta_{\eta < \omega_1} D_\eta \in F$.
For each $\eta < \omega_1$, take $B_\eta \subseteq \omega_1$ and a club $C_\eta \subseteq \omega_1$
witnessing $D_\eta \in F$.
Let $R$ be the set of all $\langle \eta , \zeta \rangle \in \omega_1 \times \omega_1$ such that $\zeta \in B_\eta$.
Take a bijection $\Gamma : \omega_1 \times \omega_1 \to \omega_1$, and let $B$ be $\Gamma [R]$
and $C'$ be the set of all $\xi < \omega_1$ with $\Gamma [ \xi \times \xi ] = \xi$.
Then $C'$ is club in $\omega_1$.
Moreover, if $\xi \in C'$ and $B \cap \xi \in K_\xi$, then $B_\eta \cap \xi \in K_\xi$ for all $\eta < \xi$.
Let $C := C' \cap \Delta_{\eta < \omega_1} C_\eta$.

Then, $B$ and $C$ witnesses that $D \in F$:
Assume $\xi \in C$ and $B \cap \xi \in K_\xi$. We must show that $\xi \in D_\eta$ for all $\eta < \xi$.
Fix $\eta < \xi$. Since $\xi \in C'$ and $B \cap \xi \in K_\xi$, we have $B_\eta \cap \xi \in K_\xi$.
Also, $\xi \in C_\eta$ since $\xi \in C$. So $\xi \in D_\eta$.
\end{proof}

A definition and a lemma below are important in $\vec{K}$-subcomplete forcings.

\begin{definition} \label{def:capture}
Suppose $\vec{K} = \langle K_\xi \mid \xi < \omega_1 \rangle$ is an adequate model sequence.
We call $\bar{M}$ a $\vec{K}$-\emph{good model} if
$\bar{M}$ is a transitive full $\mcal{L}_{\mrm{ST}}^+$-model of $\mrm{ZFC}^-$ such that
$\bar{M} \in K_{\omega_1^{\bar{M}}}$, and $\bar{M}$ is countable in $K_{\omega_1^{\bar{M}}}$.
\end{definition}

\begin{lemma} \label{lem:adequate_capture}
Let $\vec{K} = \langle K_\xi \mid \xi < \omega_1 \rangle$ be an adequate model sequence.
Suppose $M = L^A_\chi$ for some set $A$ and some regular cardinal $\chi$,
and $\mcal{H}_{\omega_1} \in M$. Let $a \in M$.
Then there are a $\vec{K}$-good model $\bar{M}$ and $\sigma : \bar{M} \prec M$
with $a \in \ran ( \sigma )$.
\end{lemma}

\begin{proof}
We may assume that $A \subseteq L_\chi [A]$.
Take a regular cardinal $\lambda > \chi$, and let $N := \langle L_\lambda [A] , {\in} \rangle$.
Note that $M \in N \models \mrm{ZFC}^-$. Moreover, $M$ is regular in $N$:

Suppose $x \in M$, $f : x \to M$ and $f \in N$. We show that $f \in M$.
Take $N' \prec N$ with $f , A , \chi \in N'$ and $\alpha := N' \cap \chi \in \chi$.
Note that $f \subseteq N'$ since $f \in N' \prec N$, and $|f|^N = |x|^N \in N' \cap \chi \subseteq N'$.
Let $\pi : N' \to N''$ be the transitive collapse.
Then $f = \pi (f) \in N''$.
Note also that $N'' = \langle L_\beta [ \pi (A) ] , {\in} \rangle$ for some $\beta < \chi$
and $\pi (A) = A \cap L_\alpha [A] \in M$. So $N'' \subseteq M$. Hence $f \in M$.

We will find $\bar{N}$ and $\tau$ such that
\begin{renumerate}
\item $\bar{N}$ is a countable transitive $\mcal{L}_{\mrm{ST}}$-model of $\mrm{ZFC}^-$,
$\bar{N} \in K_{\omega_1^{\bar{N}}}$, and $\bar{N}$ is countable in $K_{\omega_1^{\bar{N}}}$,
\item $\tau : \bar{N} \prec N$, and $a , M \in \ran ( \tau )$.
\end{renumerate}
(We do not require that $\bar{N}$ is full.)
If such $\bar{N}$ and $\tau$ are found, then $\bar{M} := \tau^{-1} ( M )$ and $\sigma := \tau \rst \bar{M}$
witnesses the lemma:
$\bar{M}$ is full since $\bar{M} = \tau^{-1} (M) \in \bar{N} \models \mrm{ZFC}^-$,
and $\bar{M}$ is regular in $\bar{N}$ by the regularity of $M$ in $N$.
Note also that $\omega_1^M = \omega_1^N = \omega_1$ since $\mcal{H}_{\omega_1} \in M$,
and so $\omega_1^{\bar{M}} = \omega_1^{\bar{N}}$. Then $\bar{M}$ and $\sigma$ are as desired clearly.

We construct $\bar{N}$ and $\tau$ satisfying (i) and (ii).
Take $N'' \prec N$ of cardinality $\omega_1$ with $\omega_1 \cup \{ a , M \} \subseteq N''$.
Take a bijection $h : \omega_1 \to N''$ with $h(0) = \{ a , M \}$, and let
$E$ be the pull-back of ${\in} \cap ( N'' \times N'' )$ by $h$.
So $h$ is an isomorphism from $J := \langle \omega_1 , E \rangle$ to $N''$.

By the adequateness of $\vec{K}$,
there are stationary many $\xi < \omega_1$ such that $J \rst \xi \in K_\xi$.
Note also that there are club many $\xi < \omega_1$ such that
$J \rst \xi \prec J$ and $h[ \xi ] \cap \omega_1 = \xi$.
So we can take $\xi > 0$ such that $J \rst \xi \in K_\xi$, $J \rst \xi \prec J$
and $h[ \xi ] \cap \omega_1 = \xi$.

Let $\rho : J \rst \xi \to \bar{N}$ be the transitive collapse.
Also, let $\tau := h \circ \rho^{-1} : \bar{N} \to N$.
We claim that $\bar{N}$ and $\tau$ are as desired.
Clearly, they satisfy (ii).

We check (i).
Note that $\bar{N} \in K_\xi$ since $J \rst \xi \in K_\xi$,
and $K_\xi$ is a transitive model of $\msf{ZFC}^-$.
Moreover $\bar{N}$ is countable in $K_\xi$ since $\xi$ is countable in $K_\xi$.
So it suffices to prove that $\xi = \omega_1^{\bar{N}}$.
For this, note that $\tau : \bar{N} \cong N \rst h[ \xi ]$ is the inverse of the transitive collapse of $N \rst h[ \xi ]$.
Then, $\omega_1^{\bar{N}} = h[ \xi ] \cap \omega_1^N = h[ \xi ] \cap \omega_1 = \xi$.
\end{proof}


\subsection{$\vec{K}$-subcomplete forcings} \label{subsec:K_subcomplete}

First, we introduce the notion of $\vec{K}$-subcomplete forcings.

As we mentioned before, it is obtained from the $\infty$-subcompleteness by restricting
$\bar{M}$ and $\bar{G}$ to those captured by $\vec{K}$.
Note that if $\bar{M}$ is a $\vec{K}$-good model, and $\bar{\mbb{P}}$ is a forcing notion in $\bar{M}$,
then there is a $\bar{\mbb{P}}$-generic filter $\bar{G} \in K_{\omega_1^{\bar{M}}}$ over $\bar{M}$.

\begin{definition} \label{def:K_subcomplete}
Suppose $\mbb{P}$ is a forcing notion and
$\vec{K} = \langle K_\xi \mid \xi < \omega_1 \rangle$ is an adequate model sequence.

For a regular cardinal $\theta$ and $a \in \mcal{H}_\theta$,
we say that $\theta$ and $a$ \emph{verify the} $\vec{K}$-\emph{subcompleteness} of $\mbb{P}$
if $\mbb{P} \in \mcal{H}_\theta$, and the following hold:
For any $A$, $\chi$, $\bar{M}$, $\bar{\mbb{P}}$, $\bar{b}$, $\sigma$, $b$ and $\bar{G}$, if
\begin{renumerate}
\item $A$ is a set, $\chi$ is an ordinal, and $\mcal{H}_\theta \subseteq L^A_\chi \models \mrm{ZFC}^-$,
\item $\bar{M}$ is a $\vec{K}$-good model, and $\bar{\mbb{P}} , \bar{b} \in \bar{M}$,
\item $\sigma : \bar{M} \prec L^A_\chi$, $a \in \ran ( \sigma )$,
and $\sigma ( \langle \bar{\mbb{P}} , \bar{b} \rangle ) = \langle \mbb{P} , b \rangle$,
\item $\bar{G}$ is a $\bar{\mbb{P}}$-generic filter over $\bar{M}$
with $\bar{G} \in K_{\omega_1^{\bar{M}}}$,
\end{renumerate}
then there is $p^* \in \mbb{P}$ which forces that there is
$\sigma^* : \bar{M} \prec L^A_\chi$ with
$\sigma^* ( \langle \bar{\mbb{P}} , \bar{b} \rangle ) = \langle \mbb{P} , b \rangle$
and $\sigma^* [ \bar{G} ] \subseteq \dot{G}$,
where $\dot{G}$ is the canonical name for a $\mbb{P}$-generic filter.

We say that $\mbb{P}$ is $\vec{K}$-\emph{subcomplete} if
there are a regular cardinal $\theta$ and $a \in \mcal{H}_\theta$
which verify the $\vec{K}$-subcompleteness of $\mbb{P}$.
\end{definition}

We make some remarks on the above definition.
\begin{aenumerate}
\item Every $\infty$-subcomplete forcing notion is $\vec{K}$-subcomplete
for any adequate model sequence $\vec{K}$.
\item Suppose $\vec{K} = \langle K_\xi \mid \xi < \omega_1 \rangle$ is an adequate model sequence
with $\mcal{H}_{\omega_1} \subseteq K_\xi$ for all $\xi < \omega_1$.
Then, the $\vec{K}$-subcompleteness is equivalent to $\infty$-subcompleteness.
\item If $\theta$ and $a$ verify the $\vec{K}$-subcompleteness of $\mbb{P}$,
then any regular cardinal $\theta ' \geq \theta$ together with $a$ verify
the $\vec{K}$-subcompleteness of $\mbb{P}$.
\item Let $\mbb{P} , A , \chi , \bar{M} , \bar{b} , \sigma , b, \bar{G} , p^*$ be
as in Definition \ref{def:K_subcomplete}. Suppose $G$ is a $\mbb{P}$-generic filter over $V$ with $p^* \in G$,
and in $V[G]$ let $\sigma^* : \bar{M} \prec L^A_\chi$ be such that $\sigma^* ( \langle \bar{\mbb{P}} , \bar{b} \rangle ) = \langle \mbb{P} , b \rangle$
and $\sigma^* [ \bar{G} ] \subseteq G$. Then, in $V[G]$, $\sigma^*$ can be extended to
$\sigma^{**} : \bar{M}[ \bar{G} ] \prec L^A_\chi [G]$ with $\sigma^{**} ( \bar{G} ) = G$
by letting $\sigma^{**} ( \dot{\bar{x}}^{\bar{G}} ) := \sigma^* ( \dot{\bar{x}} )^G$.
\end{aenumerate}

In the rest of this section, we observe basic properties of $\vec{K}$-subcomplete forcings.
First, we prove that the $\vec{K}$-subcompleteness is preserved by forcing equivalence.

\begin{lemma} \label{lem:K_subcomplete_equivalent}
Let $\vec{K}$ be an adequate model sequence.
If $\mbb{P}$ is a $\vec{K}$-subcomplete forcing notion, and $\mbb{P}'$ is a forcing notion
which is forcing equivalent to $\mbb{P}$, then $\mbb{P}'$ is $\vec{K}$-subcomplete.
\end{lemma}

\begin{proof}
It suffices to show that if there is a dense embedding between forcing notions $\mbb{P}$ and $\mbb{P}'$,
then $\mbb{P}$ is $\vec{K}$-subcomplete exactly when $\mbb{P}'$ is $\vec{K}$-subcomplete.
Suppose $\mbb{P}$ and $\mbb{P}'$ are forcing notions, and there is a dense embedding $d : \mbb{P} \to \mbb{P}'$.
We only show that if $\mbb{P}'$ is $\vec{K}$-subcomplete, then so is $\mbb{P}$.
The proof of the other direction is similar and is omitted.
Let $\vec{K} = \langle K_\xi \mid \xi < \omega_1 \rangle$.

Take $\theta$ and $a'$ which verify $\vec{K}$-subcompleteness of $\mbb{P}'$.
We may assume $a := \{ a' , \mbb{P}' , d \} \in \mcal{H}_\theta$.
We show that $\theta$ and $a$ verify the $\vec{K}$-subcompleteness of $\mbb{P}$.
Suppose $A$, $\chi$, $\bar{M}$, $\bar{\mbb{P}}$, $\bar{b}$, $\sigma$, $b$ and $\bar{G}$ satisfy
(i)--(iv) of Definition \ref{def:K_subcomplete}. We find $p^* \in \mbb{P}$ as in Definition \ref{def:K_subcomplete}.
Let $M := L^A_\chi$.

Note that $\mbb{P}' , d \in \ran ( \sigma )$ since $a \in \ran ( \sigma )$.
Let $\langle \bar{\mbb{P}}' , \bar{d} \rangle := \sigma^{-1} ( \langle \mbb{P} ' , d \rangle )$.
Then $\bar{d} : \bar{\mbb{P}} \to \bar{\mbb{P}} '$ is a dense embedding in $\bar{M}$.
Let $\bar{G}'$ be the filter on $\bar{\mbb{P}}'$ generated by $\bar{d} [ \bar{G} ]$.
Then $\bar{G}'$ is a $\bar{\mbb{P}}'$-generic filter over $\bar{M}$ with $\bar{G}' \in K_{\omega_1^{\bar{M}}}$.
Since $\theta$ and $a'$ verify the $\vec{K}$-subcompleteness of $\mbb{P}'$,
we can take $p' \in \mbb{P}'$ forcing the existence of $\sigma ' : \bar{M} \to M$ with
$\sigma ' ( \langle \bar{\mbb{P}} , \bar{b} , \bar{\mbb{P}}' , \bar{d} \rangle )
= \langle \mbb{P} , b , \mbb{P}' , d \rangle$
and $\sigma ' [ \bar{G}' ] \subseteq \dot{G}'$, where $\dot{G}'$ is the canonical name
for a $\mbb{P}'$-generic filter.
Take $p^* \in \mbb{P}$ with $d( p^* ) \leq p'$.

We show that $p^*$ is as desired.
Suppose $G$ is a $\mbb{P}$-generic filter over $V$ with $p^* \in G$.
Working in $V[G]$, we show that there is $\sigma^* : \bar{M} \prec M$
with $\sigma^* ( \langle \bar{\mbb{P}} , \bar{b} \rangle ) = \langle \mbb{P} , b \rangle$
and $\sigma^* [ \bar{G} ] \subseteq G$.
Let $G'$ be the filter on $\mbb{P}'$ generated by $d[G]$.
Then $G'$ is a $\mbb{P}'$-generic filter over $V$ with $p' \in G'$.
By the choice of $p'$, we can take $\sigma^* : \bar{M} \prec M$ with
$\sigma^* ( \langle \bar{\mbb{P}} , \bar{b} , \bar{\mbb{P}}' , \bar{d} \rangle )
= \langle \mbb{P} , b , \mbb{P}' , d \rangle$
and $\sigma^* [ \bar{G}' ] \subseteq \dot{G}'$.
Then,
\[
\sigma^* [ \bar{G} ] = \sigma^* [ \bar{d}^{-1} [ \bar{G}' ] ] \subseteq d^{-1} [G'] = G \, .
\]
So $\sigma^*$ is as desired.
\end{proof}

Next, recall that $\infty$-subcomplete forcings add no reals.
We observe that this can be generalized to $\vec{K}$-subcomplete forcings.

\begin{lemma} \label{lem:K_subcomplete_reals}
Suppose $\vec{K}$ is an adequate model sequence.
Then any $\vec{K}$-subcomplete forcing adds no reals.
In particular, every $\vec{K}$-subcomplete forcing preserves $\omega_1$.
\end{lemma}

\begin{proof}
Suppose $\mbb{P}$ is a $\vec{K}$-subcomplete forcing notion,
$p \in \mbb{P}$ and $\dot{x}$ is a $\mbb{P}$-name for a real.
It suffices to find $p^* \leq p$ forcing that $\dot{x} \in V$.

Take a regular cardinal $\theta$ and $a \in \mcal{H}_\theta$
verifying the $\vec{K}$-subcompleteness
of $\mbb{P}$. We may assume that $\omega_1 < \theta$ and $\dot{x} \in \mcal{H}_\theta$.
Take a set $A$ and a regular cardinal $\chi$ such that
$\mcal{H}_\theta \subseteq L^A_\chi$. Let $M := L^A_\chi$.

By Lemma \ref{lem:adequate_capture}, we can take a $\vec{K}$-good model $\bar{M}$
and $\sigma : \bar{M} \prec M$ such that
$a , \mbb{P} , p , \dot{x} \in \ran ( \sigma )$. Let
$\langle \bar{\mbb{P}} , \bar{p} , \dot{\bar{x}} \rangle
:= \sigma^{-1} ( \langle \mbb{P} , p , \dot{x} \rangle )$.
Since $\bar{M}$ is countable in $K_{\omega_1^{\bar{M}}}$,
we can take a $\bar{\mbb{P}}$-generic filter $\bar{G}$ over $\bar{M}$ with
$\bar{p} \in \bar{G} \in K_{\omega_1^{\bar{M}}}$. Let $x := \dot{\bar{x}}^{\bar{G}}$.

Since $\theta$ and $a$ verify the $\vec{K}$-subcompleteness of $\mbb{P}$,
we can take $p^* \in \mbb{P}$ which forces the existence of $\sigma^* : \bar{M} \prec M$
with $\sigma^* ( \langle \mbb{P} , \bar{p} , \dot{\bar{x}} \rangle ) = \langle \mbb{P} , p , \dot{x} \rangle$
and $\sigma^* [ \bar{G} ] \subseteq \dot{G}$, where $\dot{G}$ is the canonical name
for a $\mbb{P}$-generic filter.
We claim that $p^*$ is as desired.

Note that $p^* \Vdash \lchon\, p = \sigma^* ( \bar{p} ) \in \sigma^* [ \bar{G} ] \subseteq \dot{G} \,\rchon$.
So $p^* \Vdash \lchon\, p \in \dot{G} \,\rchon$. Thus $p^* \leq p$.

We prove that $p^* \Vdash \lchon\, \dot{x} \in V \,\rchon$.
Suppose $G$ is a $\mbb{P}$-generic filter over $V$ with $p^* \in G$. We show that $\dot{x}^G \in V$.
In $V[G]$, take $\sigma^* : \bar{M} \prec M$ with
$\sigma^* ( \langle \mbb{P} , \bar{p} , \dot{\bar{x}} \rangle ) = \langle \mbb{P} , p , \dot{x} \rangle$
and $\sigma^* [ \bar{G} ] \subseteq G$.
Then, $\sigma^*$ can be extended to $\sigma^{**} : \bar{M}[ \bar{G} ] \prec M[G]$
with $\sigma^{**} ( \bar{G} ) = G$.
Then, $\dot{x}^G = \sigma^{**} ( \dot{\bar{x}}^{\bar{G}} ) = \sigma^{**} (x)$.
But $\sigma^{**} (x) = x$ since $x$ is a real. So, $\dot{x}^G = x \in V$.
\end{proof}

As we have mentioned several times, Jensen \cite{Jensen3} proved that
subcomplete forcings preserve $\diamondsuit_{\omega_1}$.
Essentially the same argument proves the following lemma.

\begin{lemma} \label{lem:K_subcomplete_adequate}
If $\vec{K}$ is an adequate model sequence,
then $\vec K$ remains an adequate model sequence in any $\vec K$-subcomplete forcing extension.
If $\vec{K}$ is a $\dmnd$-model sequence,
then $\vec{K}$ remains a $\dmnd$-model sequence in any $\vec{K}$-subcomplete forcing extensions.
\end{lemma}

\begin{proof}
The latter statement follows from the former. We prove the former.

Let $\vec{K} = \langle K_\xi \mid \xi < \omega_1 \rangle$ be an adequate model sequence
and $\mbb{P}$ be a $\vec{K}$-subcomplete forcing notion.
Suppose $p \in \mbb{P}$, $\dot{B}$ and $\dot{C}$ are $\mbb{P}$-names,
and $p$ forces that $\dot{B} \subseteq \omega_1$ and $\dot{C}$ is club in $\omega_1$.
It suffices to find $p^* \leq p$ and $\xi < \omega_1$
such that $p^*$ forces $\dot{B} \cap \xi \in K_\xi$ and $\xi \in \dot{C}$.

Take a regular cardinal $\theta$ and $a \in \mcal{H}_\theta$ verifying
the $\vec{K}$-subcompleteness of $\mbb{P}$.
We may assume $\dot{B} , \dot{C} \in \mcal{H}_\theta$.
Take a set $A$ and a regular cardinal $\chi$
with $M := L^A_\chi \supseteq \mcal{H}_\theta$.

By Lemma \ref{lem:adequate_capture}, we can take a $\vec{K}$-good model $\bar{M}$
and $\sigma : \bar{M} \prec M$ such that $a , \mbb{P} , p , \dot{B} , \dot{C} \in \ran ( \sigma )$.
Let $\xi := \omega_1^{\bar{M}}$, and let
$\langle \bar{\mbb{P}} , \bar{p} , \dot{\bar{B}} , \dot{\bar{C}} \rangle
:= \sigma^{-1} ( \langle \mbb{P} , p , \dot{B} , \dot{C} \rangle )$.
Take a $\bar{\mbb{P}}$-generic filter $\bar{G}$ over $\bar{M}$ with $\bar{p} \in \bar{G} \in K_\xi$.
Then there is $p^* \in \mbb{P}$ forcing the existence of $\sigma^* : \bar{M} \prec M$ with
$\sigma^* ( \langle \bar{\mbb{P}} , \bar{p} , \dot{\bar{B}} , \dot{\bar{C}} \rangle )
= \langle \mbb{P} , p , \dot{B} , \dot{C} \rangle$
and $\sigma^* [ \bar{G} ] \subseteq \dot{G}$, where $\dot{G}$ is the canonical name
for a $\mbb{P}$-generic filter.
We claim that $p^*$ and $\xi$ are as desired.
Note that $p^* \leq p$ by the same argument as in the proof of Lemma \ref{lem:K_subcomplete_reals}.

Suppose $G$ is a $\mbb{P}$-generic filter over $V$ with $p^* \in G$.
Let $B := \dot{B}^G$ and $C := \dot{C}^G$.
In $V[G]$, we show that $B \cap \xi \in K_\xi$ and $\xi \in C$.
Take $\sigma^{**} : \bar{M}[ \bar{G} ] \prec M[G]$ with
$\sigma^{**} ( \langle \bar{\mbb{P}} , \bar{p} , \dot{\bar{B}} , \dot{\bar{C}} \rangle )
= \langle \mbb{P} , p , \dot{B} , \dot{C} \rangle$
and $\sigma^{**} ( \bar{G} ) = G$.
Let $\bar{B} := \dot{\bar{B}}^{\bar{G}}$ and $\bar{C} := \dot{\bar{C}}^{\bar{G}}$.
Then, $\sigma^{**} ( \bar{B} ) = B$, $\sigma^{**} ( \bar{C} ) = C$, and $\sigma^{**} ( \xi ) = \omega_1$.
Then, $\bar{B} = B \cap \xi$. But $\bar{B} \in K_\xi$ since $\bar{B} \in \bar{M}[ \bar{G} ] \in K_\xi$.
So $B \cap \xi \in K_\xi$.
Also, $\bar{C} = C \cap \xi$. Moreover $\bar{C}$ is unbounded in $\xi$ by the elementarity of $\sigma^*$.
Then $\xi \in C$ since $C$ is closed.
\end{proof}

Recall that all $\infty$-subcomplete forcings preserve stationary subsets of $\omega_1$.
We prove that $\vec{K}$-subcomplete forcings preserve stationary subsets of $\omega_1$
if $\vec{K}$ is strongly adequate in the sense below.

\begin{definition} \label{def:strongly_adequate}
An adequate model sequence $\vec{K} = \langle K_\xi \mid \xi < \omega_1 \rangle$ is called
a \emph{strongly adequate model sequence} if
for any $B \subseteq \omega_1$, there are club many $\xi < \omega_1$
with $B \cap \xi \in K_\xi$.
\end{definition}

\begin{lemma} \label{lem:K_subcomplete_stat_pres}
Suppose $\vec{K}$ is a strongly adequate model sequence,
and $\mbb{P}$ is a $\vec{K}$-subcomplete forcing notion.
Then $\mbb{P}$ preserves stationary subsets of $\omega_1$.
\end{lemma}

To prove this, we use the following modification of Lemma \ref{lem:adequate_capture}.

\begin{lemma} \label{lem:strongly_adequate_capture}
Let $\vec{K} = \langle K_\xi \mid \xi < \omega_1 \rangle$ be a strongly adequate model sequence.
Suppose $M = L^A_\chi$ for some set $A$ and some regular cardinal $\chi$,
$\mcal{H}_{\omega_1} \in M$, and $a \in M$.
Let $S$ be a stationary subset of $\omega_1$.
Then, there are a $\vec{K}$-good model $\bar{M}$ with $\omega_1^{\bar{M}} \in S$
and $\sigma : \bar{M} \prec M$ with $a \in \ran ( \sigma )$.
\end{lemma}

\begin{proof}
The proof is almost the same as Lemma \ref{lem:adequate_capture}.
In the fifth paragraph of the proof of Lemma \ref{lem:adequate_capture},
we took $\xi < \omega_1$ such that $J \rst \xi \in K_\xi$, $J \rst \xi \prec J$
and $h[ \xi ] \cap \omega_1 = \xi$.
Note that we can take such $\xi \in S$ under the assumption of this lemma
that $\vec{K}$ is strongly adequate and $S$ is stationary.
Then, the rest of the proof is exactly the same as that of Lemma \ref{lem:adequate_capture}.
\end{proof}

\begin{proof}[Proof of Lemma \ref{lem:K_subcomplete_stat_pres}]
Let $S$ be a stationary subset of $\omega_1$.
Suppose $p \in \mbb{P}$, $\dot{C}$ is a $\mbb{P}$-name and $p$ forces $\dot{C}$ to be
a club subset of $\omega_1$.
We find $p^* \leq p$ and $\xi \in S$ such that $p^* \Vdash \lchon\, \xi \in \dot{C} \,\rchon$.

Take a regular cardinal $\theta$ and $a \in \mcal{H}_\theta$ verifying
the $\vec{K}$-subcompleteness of $\mbb{P}$.
We may assume $\dot{C} \in \mcal{H}_\theta$.
Take a set $A$ and a regular cardinal $\chi$
with $M := L^A_\chi \supseteq \mcal{H}_\theta$.

By Lemma \ref{lem:strongly_adequate_capture}, we can take a $\vec{K}$-good model $\bar{M}$
with $\xi := \omega_1^{\bar{M}} \in S$
and $\sigma : \bar{M} \prec M$ with $a , \mbb{P} , p , \dot{C} \in \ran ( \sigma )$.
Let
$\langle \bar{\mbb{P}} , \bar{p} , \dot{\bar{C}} \rangle
:= \sigma^{-1} ( \langle \mbb{P} , p , \dot{C} \rangle )$.
Take a $\bar{\mbb{P}}$-generic filter over $\bar{M}$ with $\bar{p} \in \bar{G} \in K_\xi$.
Then, there is $p^* \in \mbb{P}$ which forces the existence of $\sigma^* : \bar{M} \prec M$ with
$\sigma^* ( \langle \bar{\mbb{P}} , \bar{p} , \dot{\bar{C}} \rangle )
= \langle \mbb{P} , p ,\dot{C} \rangle$
and $\sigma^* [ \bar{G} ] \subseteq \dot{G}$, where $\dot{G}$ is the canonical name
for a $\mbb{P}$-generic filter.

Note that $p^* \leq p$ by the same argument as in the proof of Lemma \ref{lem:K_subcomplete_reals}.
Recall also that $\xi \in S$.
Moreover, we can prove that $p^* \Vdash \lchon\, \xi \in \dot{C} \,\rchon$
by the same argument as in the proof of Lemma \ref{lem:K_subcomplete_adequate}.
So $p^*$ and $\xi$ are as desired.
\end{proof}

We have proved that $\vec{K}$-subcomplete forcings preserve stationary subsets of $\omega_1$
if $\vec{K}$ is strongly adequate.
However, this need not hold for an arbitrary adequate $\vec K$.
In fact, a $\vec{K}$-subcomplete forcing $\mbb{C}_{\vec{K} , B}$ in \S \ref{sec:K_scfa}
does not preserve stationary subsets of $\omega_1$
if $\vec{K}$ is not strongly adequate, and $B$ is its witness.
See a remark after the proof of Proposition \ref{prop:K_scfa_dmnd+}.


\section{Nice iterations of $\vec{K}$-subcomplete forcings} \label{sec:iteration}

In this section, we discuss nice iterations of $\vec{K}$-subcomplete forcings.
Here an iteration of $\vec{K}$-subcomplete forcings means the following.

\begin{definition} \label{def:it_K_subcomplete}
Let $\vec{K}$ be an adequate model sequence.
An iteration $\langle \mbb{P}_\alpha \mid \alpha \leq \delta \rangle$
is called an \emph{iteration of} $\vec{K}$-\emph{subcomplete forcings} if
\[
\Vdash_\alpha \lchon\,
\mbox{$\vec{K}$ is an adequate model sequence, and $\dot{\mbb{P}}_{\alpha , \alpha +1}$ is $\vec{K}$-subcomplete}
\,\rchon
\]
for all $\alpha < \delta$.
\end{definition}

We prove the following.

\begin{theorem} \label{thm:K_subcomplete_nice_it}
Suppose $\vec{K}$ is an adequate model sequence, and
$\langle \mbb{P}_\alpha \mid \alpha \leq \delta \rangle$ is a nice iteration of $\vec{K}$-subcomplete forcings.
Then $\mbb{P}_\delta$ is $\vec{K}$-subcomplete.
\end{theorem}

The proof of this theorem is almost the same as the analogous theorem
for $\infty$-subcomplete forcings (Fact \ref{fact:infty_subcomplete_nice_it}).
But, we give the proof for the completeness of this paper.

In \S \ref{subsec:nice_it}, we briefly review nice iterations developed by Miyamoto \cite{Miyamoto_nice}.
In \S \ref{subsec:K_subcomplete_nice_it}, we prove Theorem \ref{thm:K_subcomplete_nice_it}.


\subsection{Nice iterations} \label{subsec:nice_it}

Here we review nice iterations introduced by Miyamoto \cite{Miyamoto_nice}.
The key notions in nice iterations are those of nested antichains and their mixture.

We begin with the notion of nested antichains.
In the definition below, $S$ is essentially a tree of height $\omega$
consisting of conditions in $\bigcup \{ \mbb{P}_\alpha \mid \alpha < \delta \}$:
$S_n$ is the $n$-th level of $S$, and $\suc^n_S (s)$ is the set of immediate successors of $s \in S_n$.

\begin{definition} \label{def:nested_antichain}
Suppose that $\langle \mbb{P}_\alpha \mid \alpha < \delta \rangle$ is an iteration.
A \emph{nested antichain} in $\langle \mbb{P}_\alpha \mid \alpha < \delta \rangle$ is a pair
$S = \langle \langle S_n \mid n < \omega \rangle , \langle \suc^n_S \mid n < \omega \rangle \rangle$
such that
\begin{renumerate}
\item $S_0 = \{ s_0 \}$ for some $s_0 \in \bigcup_{\alpha < \delta} \mbb{P}_\alpha$,
\item $S_n \subseteq \bigcup_{\alpha < \delta} \mbb{P}_\alpha$ for all $n < \omega$,
\item $\suc^n_S : S_n \to \mcal{P} ( S_{n+1} )$ and
$S_{n+1} = \bigcup \{ \suc^n_S (s) \mid s \in S_n \}$ for each $n < \omega$,
\item if $s \in S_n$, and $s' \in \suc^n_S (s)$, then $l(s) \leq l(s')$,
and $s' \rst l(s) \leq s$,
\item if $s \in S_n$, then
$\langle s' \rst l(s) \mid s' \in \suc^n_S (s) \rangle$ is a maximal antichain below $s$ in $\mbb{P}_{l(s)}$.
\end{renumerate}
\end{definition}

For a nested antichain $S$,
we let $\langle S_n \mid n < \omega \rangle$ and $\langle \suc^n_S \mid n < \omega \rangle$ denote
those such that
$S = \langle \langle S_n \mid n < \omega \rangle , \langle \suc^n_S \mid n < \omega \rangle \rangle$.
We write $s \in S$ for $s \in \bigcup_{n < \omega} S_n$.
The scripts $n$ and $S$ in $\suc^n_S$ will sometimes be omitted
if they are clear from the context.

Suppose $S$ is a nested antichain in some iteration.
A unique element of $S_0$ is called a \emph{root} of $S$ and denoted as $\nart (S)$.
Suppose $n \leq n' < \omega$, $s \in S_n$ and $s' \in S_{n'}$.
We write $(s,n) \leq_S (s',n')$ if there is a sequence $\langle s_m \mid n \leq m \leq n' \rangle$
such that $s_n = s$, $s_{n'} = s'$ and $s_{m+1} \in \suc^m_S ( s_m )$ for all $m$.

Next, we recall mixtures of nested antichains.
A nested antichain $S$ can be identified with some condition in an iteration,
and such a condition is called a mixture of $S$.
Here we adopt the following definition of mixtures,
which was proved to be equivalent to the original definition in \cite[Proposition 2.5]{Miyamoto_nice}.
The following formulation may be more intuitive.

\begin{definition} \label{def:mixture}
Let $S$ be a nested antichain in an iteration $\langle \mbb{P}_\alpha \mid \alpha < \delta \rangle$
with $\nart (S) = s_0$.
For $\beta < \delta$ and $p \in \mbb{P}_\beta$,
we say that $p$ is a \emph{mixture} of $S$ up to $\beta$ if the following holds,
where $\dot{G}_\alpha$ is the canonical name for a $\mbb{P}_\alpha$-generic filter.
\begin{renumerate}
\item $p \equiv  s_0 \rst \beta$ if $\beta < l( s_0 )$,
and $p \rst l( s_0 ) \equiv s_0$ if $\beta \geq l( s_0 )$.
\item For any $s \in S$, $s \rst \beta \leq p$ if $\beta < l(s)$,
and $s \leq p \rst l(s)$ if $\beta \geq l(s)$.
\item For any $s \in S$ with $l(s) \leq \beta$ and any $s' \in \suc (s)$,
\begin{itemize}
\item $s' \rst l(s)^\frown p \rst [ l(s) , \beta ) \equiv s' \rst \beta$ if $\beta < l(s')$,
\item $s' \rst l(s)^\frown p \rst [ l(s) , l(s') ) \equiv s'$ if $\beta \geq l(s')$.
\end{itemize}
\item For any $\alpha < \beta$ and any $u \in \mbb{P}_\alpha$ with $u \leq p \rst \alpha$, if
$u$ forces the following ($\ast$), then we have
$u^\frown 1_\beta \rst [ \alpha , \beta ) \equiv u^\frown p \rst [ \alpha , \beta )$.
\begin{itemize}
\item[($\ast$)]
There is a sequence $\langle s_n \mid n < \omega \rangle \in \prod_{n < \omega} S_n$ such that
$l( s_n ) < \alpha$, $s_{n+1} \in \suc ( s_n )$, and $s_n \in \dot{G}_{l( s_n )}$
for all $n < \omega$.
\end{itemize}
\end{renumerate}
For a limit ordinal $\beta \leq \delta$ and a sequence $p$ on $\beta$,
we say that $p$ is $( S , \beta )$-\emph{nice} if
$p \rst \beta ' \in \mbb{P}_{\beta '}$, and $p \rst \beta '$ is a mixture of $S$ up to $\beta '$
for all $\beta ' < \beta$.
\end{definition}

Note that if $p \in \mbb{P}_\beta$ is a mixture of a nested antichain $S$ up to $\beta$,
then $q \in \mbb{P}_\beta$ is a mixture of $S$ up to $\beta$ if and only if $p \equiv q$.
For $p$ being $( S , \beta )$-nice, we do not require $p$ to be in $\mbb{P}_\beta$.

Now, we recall the notion of nice iterations.

\begin{definition} \label{def:nice_it}
An iteration $\langle \mbb{P}_\alpha \mid \alpha < \delta \rangle$
is called a \emph{nice iteration} if it satisfies the following conditions,
where $\dot{G}_\alpha$ is the canonical name for a $\mbb{P}_\alpha$-generic filter.
\begin{renumerate}
\item For any $\alpha < \delta$ with $\alpha + 1 \leq \delta$,
if $p \in \mbb{P}_\alpha$, and $\dot{q}$ is a $\mbb{P}_\alpha$-name such that
$p \Vdash_\alpha \lchon\, \dot{q} \in \dot{\mbb{P}}_{\alpha , \alpha + 1} \,\rchon$,
then there is $r \in \mbb{P}_{\alpha +1}$ such that
$r \rst \alpha \equiv p$ and
$p \Vdash_\alpha \lchon\, r \rst [ \alpha , \alpha + 1 ) \equiv \dot{q} \,\rchon$.
\item For any limit ordinal $\beta < \delta$, $\mbb{P}_\beta$ consists of all sequences $p$ on $\beta$
such that $p$ is $( S , \beta )$-nice for some nested antichain
$S$ in $\langle \mbb{P}_\alpha \mid \alpha < \beta \rangle$.
\end{renumerate}
\end{definition}

\cite[Lemma 2.9]{Miyamoto_nice} proved that
if $\langle \mbb{P}_\alpha \mid \alpha < \beta \rangle$ is an iteration for a limit ordinal $\beta$,
then we can extend it to an iteration $\langle \mbb{P}_\alpha \mid \alpha \leq \beta \rangle$
so that $\mbb{P}_\beta$ satisfies (ii) of the above definition.
So we can recursively construct a nice iteration as usual:
Suppose $Q$ is a class function, and
$Q( \mbb{P} )$ is a $\mbb{P}$-name of a forcing notion for every forcing notion $\mbb{P}$.
Then for any ordinal $\delta$,
we can construct a nice iteration $\langle \mbb{P}_\alpha \mid \alpha < \delta \rangle$ such that
$\Vdash_\alpha \lchon\,
\mbox{$\dot{\mbb{P}}_{\alpha , \alpha + 1}$ is forcing equivalent to $Q( \mbb{P}_\alpha )$}
\,\rchon$
for all $\alpha < \delta$.

In the rest of this subsection, we recall technical notions developed in \cite{Miyamoto_nice},
which will be used to prove Theorem \ref{thm:K_subcomplete_nice_it}.

First, we recall the notion of hooking.

\begin{definition} \label{def:hook}
Suppose $S$ and $T$ are nested antichains in some iteration.
We say that $S$ \emph{hooks} $T$ and write $S \angle\, T$
if for any $n < \omega$ and any $s \in S_n$
there is $t \in T_{n+1}$ such that $l(t) \leq l(s)$ and $s \rst l(t) \leq t$.
\end{definition}

It is easy to see that if $S$ and $T$ are nested antichains with $S \angle\, T$,
and if $p$ and $q$ are mixtures of $S$ and $T$ up to some ordinal, respectively,
then $p \leq q$. We will use the following lemma. See \cite{Miyamoto_nice} for the proof.

\begin{lemma}[{\cite[Lemma 2.11]{Miyamoto_nice}}] \label{lem:hook}
Let $\delta$ be a limit ordinal and $\langle \mbb{P}_\alpha \mid \alpha \leq \delta \rangle$ be a nice iteration.
Suppose $p,q \in \mbb{P}_\delta$, and $S, s_1$ satisfies the following.
\begin{renumerate}
\item $S$ is a nested antichain in $\langle \mbb{P}_\alpha \mid \alpha < \delta \rangle$, and $s_1 \in S_1$.
\item $p$ is a mixture of $S$ up to $\delta$.
\item $q \rst l( s_1 ) \leq s_1$, and $q \leq p$.
\end{renumerate}
Then there is a nested antichain $T$ in $\langle \mbb{P}_\alpha \mid \alpha < \delta \rangle$ such that
$q$ is a mixture of $T$ up to $\delta$ and $T \angle\, S$.
\end{lemma}

Next, we recall the notion of fusion structures.

\begin{definition} \label{def:fusion_structure}
Let $\delta$ be a limit ordinal, $\langle \mbb{P}_\alpha \mid \alpha \leq \delta \rangle$ be a nice iteration
and $S$ be a nested antichain in $\langle \mbb{P}_\alpha \mid \alpha < \delta \rangle$.
A \emph{fusion structure} of $S$ is a sequence
$\langle q^{(s,n)} , T^{(s,n)} \mid n < \omega , \, s \in S_n \rangle$ which satisfies the following properties
for all $n < \omega$ and $s \in S_n$.
\begin{renumerate}
\item $T^{(s,n)}$ is a nested antichain in $\langle \mbb{P}_\alpha \mid \alpha < \delta \rangle$.
\item $q^{(s,n)}$ is a mixture of $T^{(s,n)}$ up to $\delta$.
\item $l( \nart ( T^{(s,n)} ) ) = l(s)$, and $s \leq \nart ( T^{(s,n)} )$.
\item If $s' \in \suc^n_S (s)$, then $T^{(s' , n+1)} \angle\, T^{(s,n)}$, and so $q^{(s' , n+1)} \leq q^{(s,n)}$.
\end{renumerate}
\end{definition}

The following is a key lemma on a fusion structure.
See \cite{Miyamoto_nice} for the proof.

\begin{lemma}[{\cite[Proposition 3.5]{Miyamoto_nice}}] \label{lem:fusion}
Let $\delta$ be a limit ordinal and $\langle \mbb{P}_\alpha \mid \alpha \leq \delta \rangle$ be an iteration.
Suppose $S$ is a nested antichain in $\langle \mbb{P}_\alpha \mid \alpha < \delta \rangle$,
and $\langle q^{(s,n)} , T^{(s,n)} \mid n < \omega , \, s \in S_n \rangle$ is a fusion structure of $S$.
Assume $G_\delta$ is a $\mbb{P}_\delta$-generic filter over $V$ containing a mixture of $S$ up to $\delta$.
Then, in $V[G]$, there is $\langle s_n \mid n < \omega \rangle \in \prod_{n < \omega} S_n$
such that $s_{n+1} \in \suc^n_S ( s_n )$, $s_n \in G_\delta \rst l( s_n )$ and $q^{( s_n ,n)} \in G_\delta$ for all $n < \omega$.
\end{lemma}


\subsection{Nice iterations of $\vec{K}$-subcomplete forcings} \label{subsec:K_subcomplete_nice_it}

Here we prove Theorem \ref{thm:K_subcomplete_nice_it}.
It will be proved by induction on the length $\delta$ of the iteration
$\langle \mbb{P}_\alpha \mid \alpha \leq \delta \rangle$.
In fact, as is usual, we prove something stronger by induction.
We use the following notation.

\begin{definition} \label{def:relative_verify_K_subcomplete}
Let $\vec{K} = \langle K_\xi \mid \xi < \omega_1 \rangle$ be an adequate model sequence
and $\vec{\mbb{P}} = \langle \mbb{P}_\alpha \mid \alpha \leq \delta \rangle$ be an iteration.
For $\alpha \leq \delta$ let $\dot{G}_\alpha$ be the canonical name for a $\mbb{P}_\alpha$-generic filter.
Suppose $\theta$ is a regular cardinal with $\mbb{P}_\delta \in \mcal{H}_\theta$,
$a \in \mcal{H}_\theta$, and $\beta < \delta$.
We say that $\theta$ and $a$ \emph{verify} $\vec{K}$-\emph{subcompleteness} of $\mbb{P}_\delta$
\emph{relative to} $\beta$ if the following holds: For any
$A$, $\chi$, $b$, $\bar{M}$, $\bar{b}$,
$\vec{\bar{\mbb{P}}} = \langle \bar{\mbb{P}}_{\bar{\alpha}} \mid \bar{\alpha} \leq \bar{\delta} \rangle$,
$\bar{\beta}$, $\bar{G}_{\bar{\beta}}$, $\bar{G}_{\bar{\delta}}$ and $p$ if
\begin{renumerate}
\item $A$ is a set, $\chi$ is an ordinal, $\mcal{H}_\theta \subseteq L^A_\chi \models \mrm{ZFC}^-$,
and $b \in L^A_\chi$,
\item $\bar{M}$ is a $\vec{K}$-good model,
\item $\bar{b} , \vec{\bar{\mbb{P}}} , \bar{\beta} \in \bar{M}$,
$\vec{\bar{\mbb{P}}}$ is an iteration in $\bar{M}$, and $\bar{\beta} < \bar{\delta}$,
\item $\bar{G}_{\bar{\delta}}$ is a $\bar{\mbb{P}}_{\bar{\delta}}$-generic filter over $\bar{M}$
with $\bar{G}_{\bar{\delta}} \in K_{\omega_1^{\bar{M}}}$,
and $\bar{G}_{\bar{\beta}} = \bar{G}_{\bar{\delta}} \rst \bar{\beta}$,
\item $p \in \mbb{P}_\beta$, and $p$ forces the existence of $\sigma : \bar{M} \prec L^A_\chi$
with $a \in \ran ( \sigma )$,
$\sigma ( \langle \vec{\bar{\mbb{P}}} , \bar{\beta} , \bar{b} \rangle )
= \langle \vec{\mbb{P}} , \beta , b \rangle$
and $\sigma [ \bar{G}_{\bar{\beta}} ] \subseteq \dot{G}_\beta$,
\end{renumerate}
then there is $p^* \in \mbb{P}_\delta$ such that
$p^* \rst \beta = p$, and
$p^*$ forces the existence of $\sigma^* : \bar{M} \prec L^A_\chi$ with
$\sigma^* ( \langle \vec{\bar{\mbb{P}}} , \bar{b} \rangle ) = \langle \vec{\mbb{P}} , b \rangle$
and $\sigma^* [ \bar{G}_{\bar{\delta}} ] \subseteq \dot{G}_\delta$.
\end{definition}

We prove the following proposition by induction on $\delta$.
Note that Theorem \ref{thm:K_subcomplete_nice_it} follows from this proposition for $\alpha = 0$.

\begin{prop} \label{prop:K_subcomplete_nice_it}
Let $\vec{K}$ be an adequate model sequence and $\vec{\mbb{P}} = \langle \mbb{P}_\alpha \mid \alpha \leq \delta \rangle$
be a nice iteration of $\vec{K}$-subcomplete forcings.
Suppose $\beta < \delta$.
Then there are a regular cardinal $\theta$ and $a \in \mcal{H}_\theta$
verifying the $\vec{K}$-subcompleteness of $\mbb{P}_\delta$ relative to $\beta$.
\end{prop}

\begin{proof}
We prove the proposition by induction on the length $\delta$ of an iteration $\vec{\mbb{P}}$.
We have nothing to do for $\delta = 0$.
Suppose $\delta > 0$, and the proposition holds for all $\delta ' < \delta$.
We prove the proposition for $\delta$.

\bigskip

\noindent
\textbf{Case 1}. $\delta$ is a limit ordinal.

\smallskip

By the induction hypothesis,
for each $\alpha , \gamma$ with $\alpha < \gamma < \delta$,
take $\theta_{\alpha , \gamma}$ and $a_{\alpha , \gamma}$ which verify the $\vec{K}$-subcompleteness
of $\mbb{P}_\gamma$ relative to $\alpha$.
Let $\theta$ be a regular cardinal such that $\theta > \theta_{\alpha , \gamma}$ for all
$\alpha , \gamma$.
Also, let $a := \langle a_{\alpha , \gamma} \mid \alpha < \gamma < \delta \rangle$.
We show that $\theta$ and $a$ verify the $\vec{K}$-subcompleteness of $\mbb{P}_\delta$ relative to $\beta$.

Suppose
$A$, $\chi$, $b$, $\bar{M}$, $\bar{b}$,
$\vec{\bar{\mbb{P}}} = \langle \bar{\mbb{P}}_{\bar{\alpha}} \mid \bar{\alpha} \leq \bar{\delta} \rangle$,
$\bar{\beta}$, $\bar{G}_{\bar{\beta}}$, $\bar{G}_{\bar{\delta}}$ and $p$
satisfy (i)--(v) of Definition \ref{def:relative_verify_K_subcomplete}.
We will show that there is $p^* \in \mbb{P}_\delta$ as in Definition \ref{def:relative_verify_K_subcomplete}.
Let $M := L^A_\chi$.
For each $\bar{\alpha} < \bar{\delta}$, let $\bar{G}_{\bar{\alpha}} := \bar{G}_{\bar{\delta}} \rst \bar{\alpha}$.
Take an enumeration $\langle \bar{c}_n \mid n < \omega \rangle$ of $\bar{M}$ with
$\bar{c}_0 = \emptyset$
and an enumeration $\langle \bar{r}_n \mid n < \omega \rangle$ of $\bar{G}_{\bar{\delta}}$
such that $\bar{r}_0$ is the largest element $\bar{1}_{\bar{\delta}}$ in $\bar{\mbb{P}}_{\bar{\delta}}$.

We will construct a nested antichain $S$ in $\langle \mbb{P}_\alpha \mid \alpha < \delta \rangle$
with $\nart (S) = p$
and a fusion structure $\langle q^{(s,n)} , T^{(s,n)} \mid n < \omega , \, s \in S_n \rangle$ of $S$
together with
$\bar{q}^{(s,n)}$, $\bar{T}^{(s,n)}$, $\bar{\beta}^{(s,n)}$, $\dot{\sigma}^{(s,n)}$, $c^{(s,n)}$
for $n < \omega$ and $s \in S_n$ so that
\begin{renumerate}
\item $\bar{q}^{(s,n)} \in \bar{G}_{\bar{\delta}}$, $\bar{q}^{(s,n)} \leq \bar{r}_n$,
and $\bar{T}^{(s,n)} , \bar{\beta}^{(s,n)} \in \bar{M}$,
\item $\dot{\sigma}^{(s,n)}$ is a $\mbb{P}_{l(s)}$-name, and $s$ forces that
\begin{itemize}
\item $\dot{\sigma}^{(s,n)} : \bar{M} \prec M$, $a \in \ran ( \dot{\sigma}^{(s,n)} )$,
and $\dot{\sigma}^{(s,n)} ( \langle \vec{\bar{\mbb{P}}} , \bar{b} \rangle ) = \langle \vec{\mbb{P}} , b \rangle$,
\item $\dot{\sigma}^{(s,n)} ( \langle \bar{q}^{(u,m)} , \bar{T}^{(u,m)} , \bar{c}_m \rangle )
= \langle q^{(u,m)} , T^{(u,m)} , c^{(u,m)} \rangle$ for all $(u,m) \leq_S (s,n)$,
\item $\dot{\sigma}^{(s,n)} ( \bar{\beta}^{(s,n)} ) = l(s)$, and
$\dot{\sigma}^{(s,n)} [ \bar{G}_{\bar{\beta}^{(s,n)}} ] \subseteq \dot{G}_{l(s)}$.
\end{itemize}
\end{renumerate}

First, assuming the above objects are constructed, we show that there is
$p^*$ as desired.
Since $\vec{\mbb{P}}$ is a nice iteration, we can take a mixture $p^* \in \mbb{P}_\delta$ of $S$.
Note that $p^* \rst \beta \equiv \nart (S) = p$.
So we may assume $p^* \rst \beta = p$.
Suppose $G_\delta$ is a $\mbb{P}_\delta$-generic filter over $V$ with $p^* \in G_\delta$.
Working in $V[ G_\delta ]$, we show that there is $\sigma^* : \bar{M} \prec M$
with $\sigma^* ( \langle \vec{\bar{\mbb{P}}} , b \rangle ) = \langle \vec{\mbb{P}} , b \rangle$,
and $\sigma^* [ \bar{G}_{\bar{\delta}} ] \subseteq G_\delta$.

By Lemma \ref{lem:fusion}, we can take $\langle s_n \mid n < \omega \rangle \in \prod_{n < \omega} S_n$
such that $s_{n+1} \in \suc^n_S ( s_n )$, $s_n \in G_\delta \rst l( s_n )$ and $q^{(s_n , n)} \in G_\delta$ for all $n < \omega$.
Then, define $\sigma^* : \bar{M} \to M$ by $\sigma^* ( \bar{c}_n ) := c^{(s_n , n)}$.
We show that $\sigma^*$ is as desired.

Let $\sigma_n := ( \dot{\sigma}^{(s_n , n)} )^{G_\delta \upharpoonright l( s_n )}$
for each $n < \omega$.
Note that, $\sigma^*$ and $\sigma_n$ coincide on $\{ \bar{c}_m \mid m \leq n \}$
for all $n < \omega$.
It follows that $\sigma^* : \bar{M} \prec M$, since $\sigma_n: \bar{M} \prec M$ for all $n < \omega$,
and $\bar{M} = \{ \bar{c}_m \mid m < \omega \}$. Moreover,
since $\sigma_n ( \langle \vec{\bar{\mbb{P}}} , \bar{b} \rangle ) = \langle \vec{\mbb{P}} , b \rangle$,
we also have that
$\sigma^* ( \langle \vec{\bar{\mbb{P}}} , \bar{b} \rangle ) = \langle \vec{\mbb{P}} , b \rangle$.
Finally, note that $\sigma^* ( \bar{q}^{(s_n , n)} ) = q^{(s_n , n)}$ for all $n < \omega$.
Then
$\sigma^* ( \bar{r}_n ) \geq \sigma^* ( \bar{q}^{(s_n , n)} ) = q^{(s_n , n)} \in G_\delta$
for all $n < \omega$. So $\sigma^* [ \bar{G}_{\bar{\delta}} ] \subseteq G_\delta$
since $\bar{G}_{\bar{\delta}} = \{ \bar{r}_n \mid n < \omega \}$.

We start to construct the above objects.
By recursion on $n$, we construct $S_n$ and
$q^{(s,n)}$, $T^{(s,n)}$, $\bar{q}^{(s,n)}$, $\bar{T}^{(s,n)}$, $\bar{\beta}^{(s,n)}$, $\dot{\sigma}^{(s,n)}$,
$c^{(s,n)}$ for all $s \in S_n$.
($\suc^n_S$ will be constructed when we construct $S_{n+1}$.)
We must construct them so that they satisfy (i),(ii) above
and (iii)--(vi) below. (iii)--(vi) are properties assuring that
$S$ will be a nested antichain with $\nart (S) = p$
and that $\langle q^{(s,n)} , T^{(s,n)} \mid n < \omega , \, s \in S_n \rangle$ will be a fusion structure of $S$.
\begin{renumerate}
\addtocounter{enumi}{2}
\item $S_0 = \{ p \}$, and
$S_{n+1} = \bigcup_{s \in S_n} \suc^n_S (s) \subseteq \bigcup_{\alpha < \delta} \mbb{P}_\alpha$.
\item If $s' \in \suc^n_S (s)$, then $l(s') \geq l(s)$, and $s' \rst l(s) \leq s$. Also, for all $s \in S_n$,
$\langle s' \rst l(s) \mid s' \in \suc^n_S (s) \rangle$ is a maximal antichain below $s$ in $\mbb{P}_{l(s)}$.
\item $T^{(s,n)}$ is a nested antichain in $\langle \mbb{P}_\alpha \mid \alpha < \delta \rangle$
with $l( \nart ( T^{(s,n)} ) ) = l(s)$ and $s \leq \nart ( T^{(s,n)} )$,
and $q^{(s,n)}$ is a mixture of $T^{(s,n)}$ up to $\delta$.
\item If $s' \in \suc^n_S (s)$, then $T^{(s' , n+1)} \angle\, T^{(s,n)}$.
\end{renumerate}

First, suppose $n = 0$. Let $S_0 := \{ p \}$ and $q^{(p,0)} := 1_\delta$.
Define $T^{(p,0)}$ by $T^{(p,0)}_k := \{ 1_\beta \}$
and $\suc^k_{T^{(p,0)}} (1_\beta ) = \{ 1_\beta \}$ for all $k < \omega$.
Note that (iii) and (v) hold for $n = 0$ and $s = p$.
Next, let $\dot{\sigma}^{(p,0)}$ be a $\mbb{P}_\beta$-name of $\sigma$
in (v) of Definition \ref{def:relative_verify_K_subcomplete}.
Also, let $\bar{q}^{(p,0)} := \bar{1}_{\bar{\delta}}$,
and define $\bar{T}^{(p,0)}$ by $\bar{T}^{(p,0)}_k := \{ \bar{1}_{\bar{\beta}} \}$
and $\suc^k_{\bar{T}^{(p,0)}} ( \bar{1}_{\bar{\beta}} ) := \{ \bar{1}_{\bar{\beta}} \}$ for all $k < \omega$,
where $\bar{1}_{\bar{\beta}}$ is the largest element of $\bar{\mbb{P}}_{\bar{\beta}}$.
Let $\bar{\beta}^{(p,0)} := \bar{\beta}$ and $c^{(s,0)} := \emptyset$.
Clearly, (i) and (ii) hold for $n = 0$ and $s = p$. (iv) and (vi) are irrelevant for $n = 0$.
This completes the construction for $n = 0$.

Next, suppose $S_n$ and
$q^{(s,n)}$, $T^{(s,n)}$, $\bar{q}^{(s,n)}$, $\bar{T}^{(s,n)}$, $\bar{\beta}^{(s,n)}$,
$\dot{\sigma}^{(s,n)}$, $c^{(s,n)}$ for all $s \in S_n$ have been constructed.
We construct $\suc^n_S$, $S_{n+1}$ and
$q^{(s', n+1)}$, $T^{(s',n+1)}$, $\bar{q}^{(s',n+1)}$, $\bar{T}^{(s',n+1)}$, $\bar{\beta}^{(s',n+1)}$,
$\dot{\sigma}^{(s',n+1)}$, $c^{(s',n+1)}$ for all $s' \in S_{n+1}$.

For $s \in S_n$, let $E_s$ be the collection of all candidates of elements of $\suc^n_S (s)$.
That is, let $E_s$ be the set all $s' \in \bigcup_{\alpha < \delta} \mbb{P}_\alpha$ such that
we have the following for some $q'$, $T'$, $\bar{q}'$, $\bar{T}'$, $\bar{\beta} '$, $\dot{\sigma}'$ and $c'$.
\begin{renumerate}
\addtocounter{enumi}{6}
\item $l(s') \geq l(s)$, and $s' \rst l(s) \leq s$.
\item $T'$ is a nested antichain in $\langle \mbb{P}_\alpha \mid \alpha < \delta \rangle$
with $l( \nart (T') ) = l(s')$ and $s' \leq \nart (T')$, and $q'$ is a mixture of $T'$ up to $\delta$.
\item $T' \angle\, T^{(s,n)}$,
\item $\bar{q}' \in \bar{G}_{\bar{\delta}}$, $\bar{q}' \leq \bar{r}_{n+1}$, and
$\bar{T}' , \bar{\beta}' \in \bar{M}$.
\item $\dot{\sigma}'$ is a $\mbb{P}_{l(s')}$-name, and $s'$ forces that
\begin{itemize}
\item $\dot{\sigma}' : \bar{M} \prec M$, $a \in \ran ( \dot{\sigma} ' )$,
and $\dot{\sigma}' ( \langle \vec{\bar{\mbb{P}}} , \bar{b} \rangle ) = \langle \vec{\mbb{P}} , b \rangle$,
\item $\dot{\sigma}' ( \langle \bar{q}^{(u,m)} , \bar{T}^{(u,m)} , \bar{c}_m \rangle )
= \langle q^{(u,m)} , T^{(u,m)} , c^{(u,m)} \rangle$ for any $(u,m) \leq_S (s,n)$, and
$\dot{\sigma} ' ( \langle \bar{q}' , \bar{T}' , \bar{c}_{n+1} \rangle ) = \langle q' , T' , c' \rangle$,
\item
$\dot{\sigma} ' ( \bar{\beta}' ) = l(s')$, and $\dot{\sigma} ' [ \bar{G}_{\bar{\beta}'} ] \subseteq \dot{G}_{l(s')}$.
\end{itemize}
\end{renumerate}
For $s \in S_n$, let $D_s := \{ s' \rst l(s) \mid s' \in E_s \}$. We claim the following.

\begin{claim}
For any $s \in S_n$, $D_s$ is dense below $s$ in $\mbb{P}_{l(s)}$.
\end{claim}

\noindent
\emph{Proof of Claim 1}.
Suppose $s \in S_n$. Take an arbitrary $u \leq s$ in $\mbb{P}_{l(s)}$.
We find $s' \in E_s$ with $s' \rst l(s) \leq u$.

By (i),(ii),(v) for $n$ and $s$, in $\bar{M}$, $\bar{T}^{(s,n)}$ is a nested antichain
in $\langle \bar{\mbb{P}}_{\bar{\alpha}} \mid \bar{\alpha} < \bar{\delta} \rangle$,
and $\bar{q}^{(s,n)}$ is a mixture of $\bar{T}^{(s,n)}$.
Since $\bar{q}^{(s,n)} \in \bar{G}_{\bar{\delta}}$, we can take $\bar{t} \in \bar{T}^{(s,n)}_1$
with $\bar{t} \in \bar{G}_{l( \bar{t} )}$.
Then, we can take $\bar{q}' \in \bar{G}_{\bar{\delta}}$ such that
$\bar{q}' \leq \bar{q}^{(s,n)} , \bar{r}_{n+1}$ and $\bar{q}' \rst l( \bar{t} ) \leq \bar{t}$,
since $\bar{q}^{(s,n)} , \bar{r}_{n+1} \in \bar{G}_{\bar{\delta}}$ and $\bar{t} \in \bar{G}_{l( \bar{t} )}$.
By Lemma \ref{lem:hook}, in $\bar{M}$, we can take a nested antichain $\bar{T}'$
in $\langle \bar{\mbb{P}}_{\bar{\alpha}} \mid \bar{\alpha} < \bar{\delta} \rangle$
such that $\bar{T}' \angle\, \bar{T}^{(s,n)}$, and $\bar{q}'$ is a mixture of $\bar{T}'$
up to $\bar{\delta}$.
In $\bar{M}$, let $\bar{t}' := \nart ( \bar{T}' )$ and $\bar{\beta}' := l( \bar{t}' )$.
Note that $\bar{t} ' \in \bar{G}_{\bar{\beta}'}$ since a mixture $\bar{q}'$ of $\bar{T}'$
belongs to $\bar{G}_{\bar{\delta}}$.

Take $v \leq u$ in $\mbb{P}_{l(s)}$ and $q' , T' , c' , \beta ' \in M$ such that
\begin{renumerate}
\addtocounter{enumi}{11}
\item $v \Vdash_{l(s)} \lchon\,
\dot{\sigma}^{(s,n)} ( \langle \bar{q} ' , \bar{T} ' , \bar{c}_{n+1} , \bar{\beta} ' \rangle ) =
\langle q' , T' , c' , \beta ' \rangle
\,\rchon$.
\end{renumerate}
Note that
\begin{renumerate}
\addtocounter{enumi}{12}
\item $v \Vdash_{l(s)} \lchon\, a_{l(s) , \beta '} \in \dot{\sigma}^{(s,n)} \,\rchon$
\end{renumerate}
since $v$ forces that
$l(s) , \beta ' , a = \langle a_{\alpha , \gamma} \mid \alpha < \gamma < \delta \rangle \in \ran ( \dot{\sigma}^{(s,n)} )$.

Recall that $\theta$ and $a_{l(s) , \beta '}$ verify the $\vec{K}$-subcompleteness
of $\mbb{P}_{\beta '}$ relative to $l(s)$.
Then, by (ii),(xii),(xiii), we can take $s' \in \mbb{P}_{\beta '}$ with $s' \rst l(s) = v$
and $\mbb{P}_{\beta '}$-name $\dot{\sigma} '$ satisfying (xi).
Note that $s' \leq \nart (T')$ since $\nart ( \bar{T}' ) = \bar{t}' \in \bar{G}_{\bar{\beta}'}$,
and $s'$ forces that
$\nart (T) = \dot{\sigma}' ( \nart ( \bar{T} ' ) ) \in \dot{\sigma}' [ \bar{G}_{\bar{\beta}'} ] \subseteq \dot{G}_{l(s')}$.
Then, $s' \rst l(s) \leq u$, and
$q' , T' , \bar{q}' , \bar{T}' , \bar{\beta}' , \dot{\sigma}' , c'$ witnesses that $s' \in E_s$.
So $s'$ is as desired.
\hfill $\square$ (Claim 1)

\bigskip

For each $s \in S_n$, we construct $\suc^n_S (s)$ and
$q^{(s', n+1)}$, $T^{(s',n+1)}$, $\bar{q}^{(s',n+1)}$, $\bar{T}^{(s',n+1)}$, $\bar{\beta}^{(s',n+1)}$,
$\dot{\sigma}^{(s',n+1)}$, $c^{(s',n+1)}$ for all $s' \in \suc^n_S (s)$.
Fix $s \in S_n$. By the claim above, we can take $A_s \subseteq D_s$ which is a maximal antichain
below $s$ in $\mbb{P}_{l(s)}$.
For each $u \in A_s$, choose $s_u ' \in E_s$ with $s_u ' \rst l(s) = u$.
Then, let $\suc^n_S (s) := \{ s_u ' \mid u \in A_s \}$.
Moreover, for each $s' \in \suc^n_S (s)$,
let $q^{(s', n+1)}$, $T^{(s',n+1)}$, $\bar{q}^{(s',n+1)}$, $\bar{T}^{(s',n+1)}$, $\bar{\beta}^{(s',n+1)}$,
$\dot{\sigma}^{(s',n+1)}$, $c^{(s',n+1)}$ be $q'$, $T'$, $\bar{q}'$, $\bar{T}'$, $\bar{\beta}'$,
$\dot{\sigma}'$ and $c'$ witnessing $s' \in E_s$.

Finally, let $S_{n+1} := \bigcup_{s \in S_n} \suc^n_S (s)$.
We constructed $\suc^n_S$, $S_{n+1}$ and
$q^{(s', n+1)}$, $T^{(s',n+1)}$, $\bar{q}^{(s',n+1)}$, $\bar{T}^{(s',n+1)}$, $\bar{\beta}^{(s',n+1)}$,
$\dot{\sigma}^{(s',n+1)}$, $c^{(s',n+1)}$ for all $s' \in S_{n+1}$.
Clearly, they are as desired.

This completes the proof for Case 1.
\hfill $\square$ (Case 1)

\bigskip

\noindent
\textbf{Case 2}. $\delta$ is a successor ordinal.

\smallskip

Let $\gamma := \delta - 1$.
By the induction hypothesis, we can take a regular cardinal $\theta '$
and $a' \in \mcal{H}_{\theta '}$ which verify the $\vec{K}$-subcompleteness
of $\mbb{P}_\gamma$ relative to $\beta$. Since $1_\gamma$ forces that
$\dot{\mbb{P}}_{\gamma , \delta}$ is $\vec{K}$-subcomplete, we can also take a regular cardinal $\theta ''$
and a $\mbb{P}_\gamma$-name $\dot{a}''$ such that $1_\gamma$ forces $\theta ''$ and $\dot{a}''$ to verify
the $\vec{K}$-subcompleteness of $\dot{\mbb{P}}_{\gamma , \delta}$.
Let $\theta := \max \{ \theta ' , \theta '' \}$ and $a := \langle a' , \dot{a} '' \rangle$.
We show that $\theta$ and $a$ verify the $\vec{K}$-subcompleteness of $\mbb{P}_\delta$
relative to $\beta$.

Suppose
$A$, $\chi$, $b$, $\bar{M}$, $\bar{b}$,
$\vec{\bar{\mbb{P}}} = \langle \bar{\mbb{P}}_{\bar{\alpha}} \mid \bar{\alpha} \leq \bar{\delta} \rangle$,
$\bar{\beta}$, $\bar{G}_{\bar{\beta}}$, $\bar{G}_{\bar{\delta}}$, $p$
satisfies (i)--(v) of Definition \ref{def:relative_verify_K_subcomplete}.
We will show that there is $p^*$ as in Definition \ref{def:relative_verify_K_subcomplete}.
Let $M := L^A_\chi$, $\bar{\gamma} := \bar{\delta} -1$ and
$\bar{G}_{\bar{\gamma}} := \bar{G}_{\bar{\delta}} \rst \bar{\gamma}$.

Note that $a' \in \ran ( \sigma )$ for $\sigma$ as in (v) of Definition \ref{def:relative_verify_K_subcomplete}.
Since $\theta$ and $a'$ verify the $\vec{K}$-subcompleteness of $\mbb{P}_\gamma$ relative to $\beta$,
we can take $p' \leq \mbb{P}_\gamma$ such that $p' \rst \beta = p$ and
$p'$ forces the existence of $\sigma ' : \bar{M} \prec M$ with
$\sigma ' ( \langle \vec{\bar{\mbb{P}}} , \bar{b} \rangle ) = \langle \vec{\mbb{P}} , b \rangle$,
$a \in \ran ( \sigma ' )$
and $\sigma ' [ \bar{G}_{\bar{\gamma}} ] \subseteq \dot{G}_\gamma$.

Let $\bar{\mbb{P}}_{\bar{\gamma} , \bar{\delta}}$ be the evaluation
of $\dot{\bar{\mbb{P}}}_{\bar{\gamma} , \bar{\delta}}$ by $\bar{G}_{\bar{\gamma}}$,
and let $\bar{H}_{\bar{\gamma}} := \bar{G}_{\bar{\delta}} \rst [ \bar{\gamma} , \bar{\delta} )$.
Note that $\bar{H}_{\bar{\gamma}}$ is
a $\bar{\mbb{P}}_{\bar{\gamma} , \bar{\delta}}$-generic filter over $\bar{M}[ \bar{G}_{\bar{\gamma}} ]$.
Also, let $\dot{H}_\gamma$ be a $\mbb{P}_\gamma$-name
of the canonical name for a $\dot{\mbb{P}}_{\gamma , \delta}$-generic filter.
We claim the following.

\begin{claim}
There is $p^* \in \mbb{P}_\delta$ such that $p^* \rst \gamma = p'$ and
\[
p^* \rst \gamma \Vdash_\gamma \lchon\, p^* \rst [ \gamma , \delta ) \Vdash_{\gamma , \delta} \, \Phi \,\rchon \, ,
\]
where $\Phi$ is the statement that
there is $\tau : \bar{M}[ \bar{G}_{\bar{\gamma}} ] \prec M[ \dot{G}_\gamma ]$ with
$\tau ( \langle \vec{\bar{\mbb{P}}} , \bar{b} , \bar{G}_{\bar{\gamma}} \rangle ) =
\langle \vec{\mbb{P}} , b , \dot{G}_\gamma \rangle$
and $\tau [ \bar{H}_{\bar{\gamma}} ] \subseteq \dot{H}_\gamma$.
\end{claim}

\noindent
\emph{Proof of Claim 2}.
By (i) of Definition \ref{def:nice_it}, it suffices to prove that
\[
p' \Vdash_\gamma \lchon\,
\exists q \in \dot{\mbb{P}}_{\gamma , \delta} \; ( \, q \Vdash_{\gamma , \delta} \, \Phi \, )
\,\rchon \, .
\]
Suppose $G_\gamma$ is a $\mbb{P}_\gamma$-generic filter over $V$ with $p' \in G_\gamma$.
Let $\mbb{P}_{\gamma , \delta} := ( \dot{\mbb{P}}_{\gamma , \delta} )^{G_\gamma}$.
Working in $V[ G_\gamma ]$, we find $q \in \mbb{P}_{\gamma , \delta}$ which forces $\Phi$.

Let $a'' := ( \dot{a} '')^{G_\gamma}$.
Recall that $\theta$ and $a''$ verify the $\vec{K}$-subcompleteness of $\mbb{P}_{\gamma , \delta}$.
We want to use this to find $q$. For this, we make some preliminaries.

Since $p' \in G_\gamma$, there is $\sigma ' : \bar{M} \prec M$ with $a \in \ran ( \sigma ' )$,
$\sigma ' ( \langle \vec{\bar{\mbb{P}}} , \bar{b} \rangle ) = \langle \vec{\mbb{P}} , b \rangle$
and $\sigma ' [ \bar{G}_{\bar{\gamma}} ] \subseteq G_\gamma$.
Then, $\sigma '$ can be naturally extended
to $\sigma '' : \bar{M} [ \bar{G}_{\bar{\gamma}} ] \prec M[ G_\gamma ]$
with $\sigma '' ( \bar{G}_{\bar{\gamma}} ) = G_\gamma$.
Note that
\begin{renumerate}
\item $a'' \in \ran ( \sigma '' )$,
and $\sigma '' ( \bar{\mbb{P}}_{\bar{\gamma} , \bar{\delta}} ) = \mbb{P}_{\gamma , \delta}$.
\end{renumerate}

Let $B := ( \{ 0 \} \times A ) \cup ( \{ 1 \} \times G_\gamma )$ and $N := L^B_\chi$.
Then it is easy to check that $M[ G_\gamma ] = \langle L_\chi [A][ G_\gamma ] , {\in} , A \cap L_\chi [A] \rangle$
and $N = \langle L_\chi [B] , {\in} , B \cap L_\chi [B] \rangle$ are equivalent
in the sense that $L_\chi [A][ G_\gamma ] = L_\chi [B]$,
$A \cap L_\chi [A]$ is definable in $N$,
and $B \cap L_\chi [B]$ is definable in $M[ G_\gamma ]$.
Note also that
\begin{renumerate}
\addtocounter{enumi}{1}
\item $\mcal{H}_\theta^{V[ G_\gamma ]} \subseteq L^B_\chi \models \mrm{ZFC}^-$.
\end{renumerate}

By the elementarity of $\sigma '$, $\bar{M} = L^{\bar{A}}_{\bar{\chi}}$
for some set $\bar{A}$ and some ordinal $\bar{\chi}$.
Let $\bar{B} := ( \{ 0 \} \times \bar{A} ) \cup ( \{ 1 \} \times \bar{G}_{\bar{\gamma}} )$ and
$\bar{N} := L^{\bar{B}}_{\bar{\chi}}$. Then $\bar{M}[ \bar{G}_{\bar{\gamma}} ]$
and $\bar{N}$ are equivalent in the same sense as above.
Then, we have that
\begin{renumerate}
\addtocounter{enumi}{2}
\item $\sigma '' : \bar{N} \prec N$.
\end{renumerate}

Next, we note that
\begin{renumerate}
\addtocounter{enumi}{3}
\item $\bar{N}$ is a $\vec{K}$-good model,
\item $\bar{H}_{\bar{\gamma}}$ is a $\bar{\mbb{P}}_{\bar{\gamma} , \bar{\delta}}$-generic filter over $\bar{N}$
with $\bar{H}_{\bar{\gamma}} \in K_{\omega_1^{\bar{N}}}$.
\end{renumerate}
We only check that $\bar{N} , \bar{H}_{\bar{\gamma}} \in K_{\omega_1^{\bar{N}}}$,
and $\bar{N}$ is countable in $K_{\omega_1^{\bar{N}}}$. The other properties are easily checked.
Note that
 $\bar{N} , \bar{H}_{\bar{\gamma}} \in K_{\omega_1^{\bar{M}}}$ and
$\bar{N}$ is countable in $K_{\omega_1^{\bar{M}}}$,
since $\bar{M} , \bar{G}_{\bar{\gamma}} , \bar{G}_{\bar{\delta}} \in K_{\omega_1^{\bar{M}}}$
and $\bar{M}$ is countable in $K_{\omega_1^{\bar{M}}}$.
So it suffices to check that $\omega_1^{\bar{N}} = \omega_1^{\bar{M}}$.
For this, note that $\mbb{P}_\gamma$ preserves the adequateness of $\vec{K}$
by the definition of an iteration of $\vec{K}$-subcomplete forcings.
So $\mbb{P}_\gamma$ preserves $\omega_1$.
Then, by the elementarity of $\sigma '$, we have $\omega_1^{\bar{N}} = \omega_1^{\bar{M}}$.

Since $\theta$ and $a''$ verify the $\vec{K}$-subcompleteness of $\mbb{P}_{\gamma , \delta}$,
by (i)--(v) above, there is $q \in \mbb{P}_{\gamma , \delta}$ which forces that there is
$\tau : \bar{N} \prec N$ with
$\tau ( \langle \vec{\bar{\mbb{P}}} , \bar{b} \rangle )
= \langle \vec{\mbb{P}} , b \rangle$
and $\tau [ \bar{H}_{\bar{\gamma}} ] \subseteq \dot{H}_\gamma$.
Then $q$ is as desired.
\hfill $\square$ (Claim 2)

\bigskip

Let $p^*$ be as in Claim 2. We show that $p^*$ is as desired.
First, note that $p^* \rst \beta = p' \rst \beta = p$.
Suppose $G_\delta$ is a $\mbb{P}_\delta$-generic filter over $V$ with $p^* \in G_\delta$.
In $V[ G_\delta ]$, we find $\sigma^* : \bar{M} \prec M$
with $\sigma^* ( \langle \vec{\bar{\mbb{P}}} , \bar{b} \rangle ) = \langle \vec{\mbb{P}} , b \rangle$
and $\sigma^* [ \bar{G}_{\bar{\delta}} ] \subseteq \dot{G}_\delta$.

Note that $G_\gamma := G_\delta \rst \gamma$ is $\mbb{P}_\gamma$-generic filter over $V$,
and $H_\gamma := G_\delta \rst [ \gamma , \delta )$ is
a $( \dot{\mbb{P}}_{\gamma , \delta} )^{G_\gamma}$-generic filter over $V[ G_\gamma ]$.
Moreover, $p^* \rst \gamma \in G_\gamma$, and $p^* \rst [ \gamma , \delta ) \in H_\gamma$.
Then, by the choice of $p^*$, in $V[ G_\delta ] = V[ G_\gamma ][ H_\gamma ]$,
there is $\tau : \bar{M}[ \bar{G}_{\bar{\gamma}} ] \prec M[ G_\gamma ]$ with
$\tau  ( \langle \vec{\bar{\mbb{P}}} , \bar{b} , \bar{G}_{\bar{\gamma}} \rangle )
= \langle \vec{\mbb{P}} , b , G_\gamma \rangle$
and $\tau [ \bar{H}_{\bar{\gamma}} ] \subseteq H_\gamma$.
Let $\sigma^* := \tau \rst \bar{M}$.

Clearly, $\sigma^* : \bar{M} \prec M$ and
$\sigma^* ( \langle \vec{\bar{\mbb{P}}} , \bar{b} \rangle ) = \langle \vec{\mbb{P}} , b \rangle$.
Note that $\tau [ \bar{G}_{\bar{\gamma}} ] \subseteq G_\gamma$ and
$\tau [ \bar{H}_{\bar{\gamma}} ] \subseteq H_\gamma$.
Note also that $\bar{G}_{\bar{\delta}} = \bar{G}_{\bar{\gamma}} * \bar{H}_{\bar{\gamma}}$
and $G_\delta = G_\gamma * H_\gamma$.
So $\sigma^* [ \bar{G}_{\bar{\delta}} ] \subseteq G_\delta$.
Thus $\sigma^*$ is as desired.
\hfill $\square$ (Case 2)

\bigskip

This completes the proof of Proposition \ref{prop:K_subcomplete_nice_it}.
\end{proof}


\section{$\vec{K}$-$\msf{SCFA}$} \label{sec:K_scfa}

In this section, we study the forcing axiom for $\vec{K}$-subcomplete forcing notions.

\begin{definition} \label{def:K_scfa}
Suppose $\vec{K}$ is an adequate model sequence.
The $\vec{K}$-Subcomplete Forcing Axiom, $\vec{K}$-$\msf{SCFA}$, is the following assertion:
\begin{quote}
For any $\vec{K}$-subcomplete forcing notion $\mbb{P}$ and any family $\mcal{D}$ of dense subsets
of $\mbb{P}$ with $| \mcal{D} | \leq \omega_1$,
there is a filter $g$ on $\mbb{P}$ with $g \cap D \neq \emptyset$
for any $D \in \mcal{D}$.
\end{quote}
\end{definition}

We first prove the consistency of $\vec K$-$\msf{SCFA}$ for some $\diamondsuit$-model sequence $\vec K$.
Using Lemma \ref{lem:K_subcomplete_adequate} and Theorem \ref{thm:K_subcomplete_nice_it},
this can be proved by a standard argument.

\begin{theorem} \label{thm:con_K_scfa}
Assume there is a supercompact cardinal.
Then there is a forcing extension in which $\vec{K}$-$\msf{SCFA}$ holds for some
$\dmnd$-model sequence $\vec{K}$.
\end{theorem}

\begin{proof}
By replacing $V$ with its forcing extension by ${}^{< \omega_1} 2$,
we may assume that $\dmnd_{\omega_1}$ holds in $V$.
Take a $\dmnd$-model sequence $\vec{K}$ in $V$.

In $V$, let $\kappa$ be a supercompact cardinal, and take a Laver function $F : \kappa \to V_\kappa$.
Then, we can construct a nice iteration $\langle \mbb{P}_\alpha \mid \alpha \leq \kappa \rangle$ so that
we have the following for all $\alpha < \kappa$.
\begin{itemize}
\item If $1_\alpha$ forces that $\vec{K}$ is a $\dmnd$-model sequence,
and $F( \alpha )$ is a $\mbb{P}_\alpha$-name for a $\vec{K}$-subcomplete forcing axiom,
then $1_\alpha$ forces that $\dot{\mbb{P}}_{\alpha , \alpha +1}$ is forcing equivalent to $F( \alpha )$.
\item Otherwise, $1_\alpha$ forces that
$\dot{\mbb{P}}_{\alpha , \alpha + 1}$ is a trivial forcing notion.
\end{itemize}

Using Lemma \ref{lem:K_subcomplete_adequate} and Theorem \ref{thm:K_subcomplete_nice_it},
by induction on $\alpha \leq \kappa$, we can prove that $\mbb{P}_\alpha$ is $\vec{K}$-subcomplete,
and $\Vdash_\alpha \lchon\, \mbox{$\vec{K}$ is a $\dmnd$-model sequence} \,\rchon$.

Let $G_\kappa$ be a $\mbb{P}_\kappa$-generic filter over $V$.
Then $\vec{K}$ is a $\dmnd$-model sequence in $V[ G_\kappa ]$.
Moreover, by the standard argument used in the consistency proof of $\msf{PFA}$,
we can prove that $\vec{K}$-$\msf{SCFA}$ holds in $V[ G_\kappa ]$.
The details are standard and are omitted.
\end{proof}

We turn our attention to consequences of $\vec{K}$-$\msf{SCFA}$.
First, recall that all subcomplete forcing notions are $\vec{K}$-subcomplete.
So we have the following.

\begin{prop} \label{prop:K_scfa_scfa}
Suppose $\vec{K}$-$\msf{SCFA}$ holds for some adequate model sequence $\vec{K}$.
Then $\msf{SCFA}$ holds.
\end{prop}

Next, we prove that $\vec{K}$-$\msf{SCFA}$ for a $\dmnd$-model sequence $\vec{K}$
implies $\dmnd^+_{\omega_1}$.

\begin{prop} \label{prop:K_scfa_dmnd+}
Suppose $\vec{K}$ is an adequate model sequence,
and $\vec{K}$-$\msf{SCFA}$ holds.
Then $\vec{K}$ is strongly adequate.
If $\vec{K}$ is a $\dmnd$-model sequence in addition,
then $\vec{K}$ is a $\dmnd_{\omega_1}^+$-sequence.
\end{prop}

For this, we use the following forcing notion.

\begin{definition} \label{def:C_KB}
For an adequate model sequence $\vec{K} = \langle K_\xi \mid \xi < \omega_1 \rangle$
and $B \subseteq \omega_1$, let $\mbb{C}_{\vec{K} , B}$ be the following forcing notion.
\begin{renumerate}
\item $\mbb{C}_{\vec{K} , B}$ consists of all closed bounded $p \subseteq \omega_1$
such that $B \cap \xi , p \cap \xi \in K_\xi$ for all $\xi \in p$.
\item $p \leq q$ in $\mbb{C}_{\vec{K} , B}$ if $p$ is an end-extension of $q$.
\end{renumerate}
\end{definition}

In the following two lemmas, we observe basic properties of $\mbb{C}_{\vec{K} , B}$.

\begin{lemma} \label{lem:C_KB_unbounded}
Suppose $\vec{K} = \langle K_\xi \mid \xi < \omega_1 \rangle$ is an adequate model sequence,
and $B \subseteq \omega_1$. Then, $D_\xi := \{ p \in \mbb{C}_{\vec{K} , B} \mid \max (p) \geq \xi \}$
is dense in $\mbb{C}_{\vec{K} , B}$ for any $\xi < \omega_1$.
\end{lemma}

\begin{proof}
Suppose $p \in \mbb{C}_{\vec{K} , B}$ and $\xi < \omega_1$.
We must find $q \leq p$ with $q \in D_\xi$.
By the adequateness of $\vec{K}$, there is $\zeta < \omega_1$
such that $\xi , \max (p) < \zeta$ and $B \cap \zeta , p \in K_\zeta$.
Let $q := p \cup \{ \zeta \}$. Then, $q \in \mbb{C}_{\vec{K} , B}$ by the choice of $\zeta$
and the fact that $p \in \mbb{C}_{\vec{K} , B}$. Moreover $q \leq p$ clearly,
and $q \in D_\xi$ since $\max (q) = \zeta > \xi$.
\end{proof}

\begin{lemma} \label{lem:C_KB_K_subcomplete} 
Suppose $\vec{K} = \langle K_\xi \mid \xi < \omega_1 \rangle$ is an adequate model sequence,
and $B \subseteq \omega_1$. Then $\mbb{C}_{\vec{K} , B}$ is $\vec{K}$-subcomplete.
\end{lemma}

\begin{proof}
Let $\mbb{P} := \mbb{C}_{\vec{K} , B}$.
Let $\theta$ be a sufficiently large regular cardinal, and let $a := B$.
We show that $\theta$ and $a$ verify the $\vec{K}$-subcompleteness of $\mbb{P}$.
Suppose $A$, $\chi$, $\bar{M}$, $\bar{\mbb{P}}$, $\bar{b}$, $\sigma$, $b$ and $\bar{G}$
satisfy (i)--(iv) of Definition \ref{def:K_subcomplete}.
It suffices to find $p^* \in \mbb{P}$ which forces that $\sigma [ \bar{G} ] \subseteq \dot{G}$,
where $\dot{G}$ is the canonical name for a $\mbb{P}$-generic filter.
(Then $p^*$ forces that $\sigma^* := \sigma$ satisfies the requirements in Definition \ref{def:K_subcomplete}.)

Let $\xi := \omega_1^{\bar{M}}$ and $\bar{B} := \sigma^{-1} (B)$.
Note that $\bar{B} = B \cap \xi$ since the critical point of $\sigma$ is $\xi$, and $\sigma ( \xi ) = \omega_1$.
Note also that $\sigma \rst \bar{\mbb{P}}$ is an identity since $\mbb{P} \subseteq \mcal{H}_{\omega_1}$.
In particular, $\sigma [ \bar{G} ] = \bar{G}$.

Let $p' := \bigcup \bar{G}$. Then $p'$ is club in $\xi$ by Lemma \ref{lem:C_KB_unbounded}.
Moreover, $B \cap \eta , p' \cap \eta \in K_\eta$ for all $\eta \in p'$
since $\bar{G} = \sigma [ \bar{G} ] \subseteq \mbb{P}$.
Note also that $B \cap \xi = \bar{B} \in \bar{M} \subseteq K_\xi$ and
$p' \in K_\xi$. Then $p^* := p' \cup \{ \xi \} \in \mbb{P}$, and $p^*$ is a lower bound of
$\bar{G} = \sigma[ \bar{G} ]$. So $p^*$ forces that $\sigma [ \bar{G} ] \subseteq \dot{G}$.
\end{proof}

Now, we prove Proposition \ref{prop:K_scfa_dmnd+}.

\begin{proof}[Proof of Proposition \ref{prop:K_scfa_dmnd+}]
We only prove the latter statement. The proof of the former is the same.
Let $\vec{K} = \langle K_\xi \mid \xi < \omega_1 \rangle$ be a $\dmnd$-model sequence,
and suppose $\vec{K}$-$\msf{SCFA}$ holds. We show that $\vec{K}$ is a $\dmnd^+_{\omega_1}$-sequence.

Take an arbitrary $B \subseteq \omega_1$. We find a club $C \subseteq \omega_1$
with $B \cap \xi , C \cap \xi \in K_\xi$ for all $\xi \in C$.
For each $\xi < \omega_1$, let $D_\xi := \{ p \in \mbb{C}_{\vec{K} , B} \mid \max (p) \geq \xi \}$.
By Lemma \ref{lem:C_KB_unbounded}, each $D_\xi$ is dense in $\mbb{C}_{\vec{K} , B}$.
By $\vec{K}$-$\msf{SCFA}$,
we can take a filter $g$ on $\mbb{C}_{\vec{K} , B}$ with $g \cap D_\xi \neq \emptyset$
for any $\xi < \omega_1$.
Let $C := \bigcup g$.

Then $C$ is a club subset of $\omega_1$.
Moreover $B \cap \xi , C \cap \xi \in K_\xi$ for all $\xi \in C$ since $g \subseteq \mbb{C}_{\vec{K} , B}$.
So $C$ is as desired.
\end{proof}

Here we make a remark on the preservation of stationary subsets of $\omega_1$
by $\vec{K}$-subcomplete forcings.
If $\vec{K}$ is an adequate model sequence, and $\vec{K}$-$\msf{SCFA}$ holds,
then $\vec{K}$-subcomplete forcings preserve
stationary subsets of $\omega_1$ by Lemma \ref{lem:strongly_adequate_capture}
and Proposition \ref{prop:K_scfa_dmnd+}.
However, $\vec K$-subcomplete forcings need not preserve stationary subsets of $\omega_1$ in general.
Suppose $\vec{K}$ is not strongly adequate, and let $B$ be a subset of $\omega_1$
such that $S := \{ \xi < \omega_1 \mid B \cap \xi \notin K_\xi \}$ is stationary.
$\mbb{C}_{\vec{K} , B}$ is $\vec{K}$-subcomplete by Lemma \ref{lem:C_KB_K_subcomplete}.
But $\mbb{C}_{\vec{K} , B}$ adds a club $C \subseteq \omega_1$ with $C \cap S = \emptyset$
by Lemma \ref{lem:C_KB_unbounded}.

By Theorem \ref{thm:con_K_scfa} and Proposition \ref{prop:K_scfa_scfa} and \ref{prop:K_scfa_dmnd+},
we have the following corollary.

\begin{cor} \label{cor:con_scfa_dmnd+}
Assume there is a supercompact cardinal.
Then there is a forcing extension in which $\msf{SCFA}$ and $\dmnd^+_{\omega_1}$ hold.
\end{cor}


\end{document}